\documentclass[reqno,a4paper, 11pt]{amsart}

\usepackage[a4paper=true,pdfpagelabels]{hyperref}
\usepackage{graphicx}

\usepackage[ansinew]{inputenc}
\usepackage{amsfonts,epsfig}
\usepackage{latexsym}
\usepackage{amsmath}
\usepackage{amssymb}
\usepackage{mathabx}

\newcommand{\DD}{\widehat{\mathcal{D}}}
\newcommand{\DDD}{\mathcal{D}}
\newcommand{\Dd}{\widecheck{\mathcal{D}}}
\newcommand{\Apqo}{A^{p,q}_\omega}
\newcommand{\lpq}{\ell^{p,q}}

\newcommand{\D}{\mathbb{D}}

\newcommand{\N}{\mathbb{N}}

\newcommand{\C}{\mathbb{C}}

\renewcommand{\phi}{\varphi}

\newcommand{\T}{\mathbb{T}}

\newcommand{\whw}{\widehat{\omega}}

\def\a{\alpha}       \def\b{\beta}        \def\g{\gamma}
           
     \def\om{\omega}      
              
                  \def\z{\zeta}
                  \def\vp{\varphi}

\def\omg{\widehat{\omega}}

\renewcommand{\H}{\mathcal{H}}

\newtheorem{theorem}{Theorem}
\newtheorem{lemma}[theorem]{Lemma}

\newtheorem{lettertheorem}{Theorem}
\newtheorem{letterlemma}[lettertheorem]{Lemma}

\theoremstyle{definition}

\theoremstyle{remark}

\numberwithin{equation}{section}

\newenvironment{Prf}{\noindent{\emph{Proof of}}}
{\hfill$\Box$ }

\begin{document}

\title[Atomic decomposition and Carleson measures for weighted mixed norm spaces]{Atomic decomposition and Carleson measures for weighted mixed norm spaces}


\author{Jos\'e \'Angel Pel\'aez}
\address{Departamento de An\'alisis Matem\'atico, Universidad de M\'alaga, Campus de
Teatinos, 29071 M\'alaga, Spain} \email{japelaez@uma.es}

\author{Jouni R\"atty\"a}
\address{University of Eastern Finland, P.O.Box 111, 80101 Joensuu, Finland}
\email{jouni.rattya@uef.fi}

\author{Kian Sierra}
\address{Departamento de An\'alisis Matem\'atico, Universidad de M\'alaga, Campus de
Teatinos, 29071 M\'alaga, Spain\\
\newline University of Eastern Finland, P.O.Box 111, 80101 Joensuu, Finland}
\email{kiansierra@hotmail.com}

\thanks{This research was supported in part by Ministerio de Econom\'{\i}a y Competitivivad, Spain, projects
MTM2014-52865-P and MTM2015-69323-REDT; La Junta de Andaluc{\'i}a,
project FQM210; Academy of Finland project no. 268009.}

\begin{abstract}
The purpose of this paper is to establish an atomic decomposition for functions in the weighted mixed norm space $A^{p,q}_\omega$ induced by a radial weight $\omega$ in the unit disc admitting a two-sided doubling condition. The obtained decomposition is further applied to characterize Carleson measures for $A^{p,q}_\omega$, and bounded differentiation operators $D^{(n)}(f)=f^{(n)}$ acting from $A^{p,q}_\omega$ to $L^p_\mu$, induced by a positive Borel measure $\mu$, on the full range of parameters $0<p,q,s<\infty$.
\end{abstract}

\maketitle

\section{Introduction and main results}

Let $\H(\D)$ denote the space of all analytic functions in the open unit disc $\D=\{z\in\C: |z|<1\}$ of the complex
plane $\C$. Further, let $\T$ stand for the boundary of $\D$ and $D(a,r)=\left\{z: |z-a|<r\right\}$ for the
Euclidean disc of center $a\in\C$ and radius $r>0$. For $0<r<1$ and $f\in \H (\D)$, set
    \begin{equation*}
    \begin{split}
    M_p(r,f)&=\left(\frac{1}{2\pi}\int_{0}^{2\pi} |f(re^{it})|^p\,dt\right)^{1/p}, \quad 0<p<\infty,\\
    M_\infty(r,f)&=\sup_{\vert z\vert =r}|f(z)|.
    \end{split}
    \end{equation*}
An integrable function $\om:\D\to[0,\infty)$ is called a weight. It is radial if $\om(z)=\om(|z|)$ for all $z\in\D$.
For a radial weight $\om$, write $\widehat{\om}(z)=\int_{|z|}^1\om(s)\,ds$ for all $z\in\D$.

For $0<p\le\infty$, $0<q<\infty$ and a radial weight $\om$, the weighted mixed norm space~$A^{p,q}_\omega$ consists of $f\in\H(\D)$ such that
    $$
    \|f\|_{A^{p,q}_\omega}^q=\int_{0}^1 M^q_p(r,f)\om(r)\,dr<\infty.
    $$
If $q=p$, then $\Apqo$ coincides with the Bergman space $A^p_\om$ induced by the weight $\omega$.
As usual, $A^p_\alpha$ denotes the weighted Bergman space induced by the standard radial weight $(1-|z|^2)^\alpha$. Weighted mixed norm spaces arise naturally in operator and function theory, for example, in the study of the boundedness, compactness and Schatten classes of the generalized Hilbert operator $H_g(f)(z)=\int_{0}^1f(t)g'(tz)\,dt$ acting on Bergman spaces~\cite{PelRathg,PelSeco}.

A weight $\om$ belongs to the class~$\DD$ if there exists a constant $C=C(\om)\ge1$ such that $\widehat{\om}(r)\le C\widehat{\om}(\frac{1+r}{2})$ for all $0\le r<1$. Moreover, if there exist $K=K(\om)>1$ and $C=C(\om)>1$ such that
    \begin{equation}\label{K}
    \widehat{\om}(r)\ge C\widehat{\om}\left(1-\frac{1-r}{K}\right),\quad 0\le r<1,
    \end{equation}
then we write $\om\in\Dd$. Weights $\om$ belonging to $\DDD=\DD\cap\Dd$ are called doubling. The classes of weights $\DD$ and $\DDD$ emerge from fundamental questions in operator theory: recently the first two authors showed that the weighted Bergman projection $P_\om$, induced by a radial weight $\om$, is bounded from $L^\infty$ to the Bloch space
$\mathcal{B}=\{f\in \H(\D):\, \sup_{z\in\D}|f'(z)|(1-|z|)<\infty\}$ if and only if $\om\in\DD$, and further, it is
bounded and onto if and only if $\om\in\DDD$~\cite{PR2018}.

The primary aim of this study is to establish a representation theorem, commonly known as an atomic decomposition, for functions in $A^{p,q}_\om$ in the sense of Coifman and Rochberg~\cite{CR}. This last-mentioned celebrated result concerning classical weighted Bergman spaces has been extended to the vector-valued Bergman spaces \cite{CJOT2008}, the Bergman spaces induced by exponential weights \cite{Arrousithesis}, the classical Dirichlet spaces \cite{GaGiPeTams2011}, the Fock spaces \cite{Zhu2012} and the classical mixed norm spaces on the upper half plane~\cite{RicciTaibleson1983}. In concrete means we will prove that each function in the mixed norm space $A^{p,q}_\om$ with $\om\in\DDD$ can be written as an adequate sum of normalized translates and dilates of powers of the Cauchy kernel in such a way that the coefficients belong to the doubled indexed complex-valued sequence space $\ell^{p,q}$.
For $0<p,q\le\infty$, the space $\ell^{p,q}$ consists sequences $\lambda=\{\lambda_{j,l}\}_{j,l}$ such that
    $$
    \left\| \lambda \right\|_{\ell^{p,q}}
    =\left\|\left\{\left\|\left\{\lambda_{j,l}\right\}_l\right\|_{\ell^p}\right\}_j\right\|_{\ell^q}<\infty,
    $$
where $\|\{a_n\}_n\|_{\ell^\infty}=\sup_n|a_n|$ and $\|\{a_n\}_n\|_{\ell^s}^s=\sum_n|a_n|^s$ for all $0<s<\infty$.

In order to state our main results we need to introduce some notation and recall that the class $\DD$
can be described by the equivalent conditions given in the following lemma~\cite[Lemma~2.1]{PelSum14}.

\begin{letterlemma}\label{Lemma:replacement-Lemmas-Memoirs}
Let $\om$ be a radial weight. Then the following conditions are
equivalent:
\begin{itemize}
\item[\rm(i)] $\om\in\DD$;
\item[\rm(ii)] There exist $C=C(\om)>0$ and $\b=\b(\om)>0$ such that
    \begin{equation*}
    \begin{split}
    \widehat{\om}(r)\le C\left(\frac{1-r}{1-t}\right)^{\b}\widehat{\om}(t),\quad 0\le r\le t<1;
    \end{split}
    \end{equation*}
\item[\rm(iii)] There exist $C=C(\om)>0$ and $\gamma=\gamma(\om)>0$ such that
    \begin{equation*}
    \begin{split}
    \int_0^t\left(\frac{1-t}{1-s}\right)^\g\om(s)\,ds
    \le C\widehat{\om}(t),\quad 0\le t<1.
    \end{split}
    \end{equation*}
\end{itemize}
\end{letterlemma}

We write $\varrho(a,z)=|\vp_a(z)|=\left|\frac{a-z}{1-\overline{a}z}\right|$
for the pseudohyperbolic distance between $z$ and $a$, and
$\Delta(a,r)=\{z:\varrho(a,z)<r\}$ for the pseudohyperbolic disc of
center $a\in\D$ and radius $r\in(0,1)$. A sequence
$\{z_k\}_{k=0}^\infty$ in $\D$ is called separated if $\inf_{k\ne j}\varrho(z_k,z_j)>0$.
Now for each $K>1$, a sequence $\{z_k\}$ in $\D$, is re-indexed in the following way depending on $K$:
For each $j\in\N\cup\{0\}$, let $\{z_{j,l}\}_l$ denote the points of the sequence $\{z_k\}$ in the annulus $A_j=A_j(K)=\{z:r_j\le|z|<r_{j+1}\}$, where $r_j=r_j(K)=1-K^{-j}$. The following result contains a half of the aforementioned atomic decomposition for functions in $\Apqo$.

\begin{theorem}\label{th:atomicdecomx}
Let $0<p\le \infty$, $0<q<\infty$, $1<K<\infty$, $\omega\in \DDD$, and $\{z_k\}_{k=0}^\infty$ a separated sequence in $\D$.
Let $\beta=\b(\om)>0$ and $\gamma=\gamma(\om)>0$ be those of Lemma~\ref{Lemma:replacement-Lemmas-Memoirs}(ii) and (iii).
If
    \begin{equation}\label{eq:M}
    M>1+\frac{1}{p}+\frac{\b+\gamma}{q}
    \end{equation}
and $\lambda=\{\lambda_{j,l}\}\in \ell^{p,q}$, then the function $F$ defined by
    $$
    F(z)=\sum_{j,l}\lambda_{j,l}\frac{(1-|z_{j,l}|)^{M-\frac{1}{p}}\omg(z_{j,l})^{-\frac{1}{q}}}{\left(1-\overline{z_{j,l}}z\right)^M}
    $$
belongs to $\H(\D)$, and there exists a constant $C=C(K,M,\om,p,q)>0$ such that
    \begin{equation}\label{eq:Fnormcontrol}
    \Vert F\Vert_{A^{p,q}_\omega}\leq C\left\|\lambda\right\|_{\ell^{p,q}}.
    \end{equation}
\end{theorem}

One important tool in the proof of Theorem~\ref{th:atomicdecomx}, to be given in Section~\ref{sec2}, is the description due to Muckenhoupt~\cite{Muck} of the weights $U$ and $V$ such that the Hardy operators $\int_0^x f(t)\,dt$ and $\int_x^\infty f(t)\,dt$ are bounded from the Lebesgue space $L^s(U^s, (0,\infty))$ to $L^s(V^s, (0,\infty))$.

To complete the atomic decomposition we are after, for each $K\in\N\setminus\{1\}$, $j\in\N\cup\{0\}$ and $l=0,1,\ldots,K^{j+3}-1$, define the dyadic polar rectangle as
    \begin{equation}\label{eq:qjl}
    Q_{j,l}=\left\{z\in \D:r_j\le |z|<r_{j+1}, \arg \,z\in \left[2\pi\frac{l}{K^{j+3}},2\pi\frac{l+1}{K^{j+3}}\right)\right\},
    \end{equation}
where $r_j=r_j(K)=1-K^{-j}$ as before, and denote its center by $\z_{j,l}$. For each $M\in\N$ and $k=1,\ldots,M^2$, the rectangle $Q_{j,l}^k$ is defined as the result of dividing $Q_{j,l}$ into $M^2$ pairwise disjoint rectangles of equal Euclidean area, and the centers of these squares are denoted by $\z_{j,l}^k$, respectively. Write $\lambda=\{\lambda_{j,l,k}\}\in\lpq$ if
    $$
    \Vert \lambda \Vert_{\ell^{p,q}}=\left(\sum_{j=0}^\infty \left(\sum_{l=0}^{K^{j+3}-1} \sum_{k=1}^{M^2}|\lambda_{j,l,k}|^p \right)^\frac{q}{p} \right)^\frac{1}{q}<\infty.
    $$
The representation part of our result reads as follows and will be proven in Section~\ref{sec3}.

\begin{theorem}\label{th:reverse}
Let $0<p\le\infty$, $0<q<\infty$, $K\in\N\setminus\{1\}$ and $\omega\in \DDD$ such that \eqref{K} holds. Then
there exists $M=M(p,q,\omega)>0$ such that $\Apqo$ consists of functions of the form
    \begin{equation}\label{Reveq}
    f(z)=\sum_{j,l,k} \lambda(f)_{j,l}^k \frac{(1-|\z_{j,l}^k|^2)^{M-\frac1p}\widehat{\om}(r_j)^{-\frac1q}}{(1-\overline{\z_{j,l}^k}z)^{M}},\quad z\in\D,
    \end{equation}
where $\lambda(f)=\{\lambda(f)_{j,l}^k\}\in\ell^{p,q}$ and
    \begin{equation}\label{eq:Fnormcontro2l}
   \left\|\{\lambda(f)_{j,l}^k\} \right\|_{\ell^{p,q}}\asymp\Vert f\Vert_{A^{p,q}_\omega}.
    \end{equation}
\end{theorem}

Atomic decompositions, and even partial results of the same fashion, for functions in spaces of analytic functions are very useful in operator theory.
In particular, they can be used to describe dual spaces \cite{RicciTaibleson1983} or to study basic questions such as the boundedness, the compactness or the Schatten class membership of concrete operators
\cite{AlCo,Arrousithesis,OC,GaGiPeTams2011,PP,JAJR,JAJREMB,Zhu,Zhu2012}. In this study we will use Theorem~\ref{th:atomicdecomx} to describe those positive Borel measures $\mu$ on $\D$ such that the differentiation operator defined by $D^{(n)}(f)=f^{(n)}$ for $n\in\N\cup\{0\}$ is bounded from $A^{p,q}_\om$ to the Lebesgue space $L^s_\mu$. The special case $n=0$ gives a description of the $s$-Carleson measures for $A^{p,q}_\om$. Carleson measures have attracted a lot of attention during the last decades because of their numerous applications in the operator theory and elsewhere, and descriptions of these measures have been obtained for many spaces of analytic functions such as the Hardy spaces \cite{CarlesonL58,CarlesonL62,Duren,Durenhp,Lu90}, the classical Bergman spaces \cite{Luecking,Zhu}, the Bergman spaces induced by Bekoll\'e-Bonami, rapidly decreasing or doubling weights \cite{OC,PP,JAJR,JAJREMB}, the Fock spaces \cite{Zhu2012} and the classical mixed norm spaces \cite{Luecking}, to name a few instances.

To state our result on the differentiation operator, write
    $$
    T_{r,u,v}(z)=\frac{\mu(\Delta(z,r))}{(1-|z|)^{u}\whw(z)^v},\quad z\in\D,
    $$
for a positive Borel measure $\mu$ on $\D$, $0<r<1$ and $0<u,v<\infty$.

\begin{theorem}\label{Theorem:Carleson}
Let $0<p,q,s<\infty$, $n\in\N\cup\{0\}$, $\omega\in\DDD$, $0<r<1$, $\mu$ a positive Borel measure on $\D$, and let $K=K(\om)\in\N\setminus\{1\}$ such that \eqref{K} holds. Then the following statements are equivalent:
    \begin{enumerate}
    \item[(i)] $D^{(n)}:\Apqo \to L^s_\mu$ is bounded;
    \item[(ii)] $\displaystyle\left\{\mu(Q_{j,l})K^{sj\left(n+\frac{1}{p}\right)}\whw(r_j)^{-\frac{s}{q}}\right\}_{j,l}\in\ell^{\left(\frac{p}{s}\right)',\left(\frac{q}{s}\right)'}$;
    \item[(iii)] $T_{r,u,v}\in L^{\left(\frac{p}{s}\right)',\left(\frac{q}{s}\right)'}_\omega$, where
\begin{itemize}
\item[(a)] $u=sn+1$ and $v=1$ if $s<\min\{p,q\}$;
\item[(b)] $u=sn+\frac1p$ and $v=1$ if $p\leq s<q$;
\item[(c)] $u=sn+1$ and $v=\frac{s}{q}$ if $q\leq s<p$;
\item[(d)] $u=s(n+\frac1p)$ and $v=\frac{s}{q}$ if $s\ge\max\{p,q\}$.
\end{itemize}
    \end{enumerate}
Moreover,
    $$
    \Vert D^{(n)} \Vert^s_{\Apqo \to L^s_\mu}
    \asymp\left\|\left\{\mu(Q_{j,l})K^{sj\left(n+\frac{1}{p}\right)}\whw(r_j)^{-\frac{s}{q}}\right\}_{j,l}\right\|_{\ell^{\left(\frac{p}{s}\right)',\left(\frac{q}{s}\right)'}}
    \asymp\|T_{r,u,v}\|_{L^{\left(\frac{p}{s}\right)',\left(\frac{q}{s}\right)'}_\omega}.
    $$
\end{theorem}

Theorem~\ref{Theorem:Carleson} will be proven in Section~\ref{sec5}.

\section{Proof of Theorem~\ref{th:atomicdecomx}}\label{sec2}

Throughout the proof and in several other occasions in this work we will use the fact that a radial weight $\om$ belongs to $\Dd$ if and only if
there exist $C=C(\om)>0$ and $\a=\a(\om)>0$ such that
    \begin{equation}  \label{Lemma:weights-in-R}
    \begin{split}
    \widehat{\om}(t)\le C\left(\frac{1-t}{1-r}\right)^{\a}\widehat{\om}(r),\quad 0\le r\le t<1.
    \end{split}
    \end{equation}
This equivalence can be proved by following the ideas used in the proof of \cite[Lemma~2.1]{PelSum14}.

To see that the function $F$ defined in Theorem~\ref{th:atomicdecomx} is analytic, observe first that $\#\{z_k\in A_j\}\lesssim K^j$ for all $j\in\N\cup\{0\}$ since $\{z_k\}$ is a separated sequence by the hypothesis. This together with Lemma~\ref{Lemma:replacement-Lemmas-Memoirs}(ii) and the hypothesis \eqref{eq:M} yields
    \begin{equation*}
    \begin{split}
 	|F(z)|\le
  	&\sum_{j,l}|\lambda_{j,l}|\frac{(1-|z_{j,l}|)^{M-\frac{1}{p}}\omg(z_{j,l})^{-\frac{1}{q}}}{|1-\overline{z_{j,l}}z|^{M}}
  	\lesssim \frac{\Vert \lambda \Vert_{\ell^\infty}}{(1-r)^M} \sum_{j=0}^\infty K^{-j\left(M-\frac{1}{p}-1\right)}\omg(r_j)^{-\frac{1}{q}}\\
    &\lesssim\frac{\Vert \lambda \Vert_{\ell^\infty}}{\omg(0)^{\frac{1}{q}}(1-r)^M} \sum_{j=0}^\infty K^{-j\left(M-\frac{1}{p}-1-\frac{\b}{q}\right)}
  	\asymp\frac{\Vert \lambda \Vert_{\ell^\infty}}{(1-r)^M},\quad |z|<r<1,
  	\end{split}
    \end{equation*}
and hence $F\in\H(\D)$.

From now on we write $c^p_j=\sum_l|\lambda_{j,l}|^p$. By following the idea used in the half plane case~\cite[(1.5) Theorem]{RicciTaibleson1983}, we will split the proof into six cases according to the values of $p$ and $q$:
\begin{enumerate}
\item[\rm Case] 1.1: $\, 0<p\le 1$ and $q\le p$;
\item[\rm Case] 1.2: $\, 0<p\le 1$ and $p<q$;
\item[\rm Case] 2.1: $\, 1<p<\infty$ and $q\le p$;
\item[\rm Case] 2.2: $\, 1<p<\infty$ and $p<q$;
\item[\rm Case] 3.1: $\, p=\infty$ and $0<q\le1$;
\item[\rm Case] 3.2: $\, p=\infty$ and $1<q<\infty$.
\end{enumerate}

Assume first $0<p\le1$. Then
    \begin{equation*}
    \begin{split}
  	|F(z)|^p\le
  	&\sum_{j,l}|\lambda_{j,l}|^p\frac{(1-|z_{j,l}|)^{pM-1}\omg(z_{j,l})^{-\frac{p}{q}}}{\left|1-\overline{z_{j,l}}z\right|^{pM}},\quad z\in\D.
    \end{split}
    \end{equation*}
By using $pM>1$, which follows from the hypothesis \eqref{eq:M}, and Lemma~\ref{Lemma:replacement-Lemmas-Memoirs}(ii)
we obtain
    \begin{equation}
    \begin{split}\label{eq:Fpmenor1}
  	M^p_p(r,F)
  	&\lesssim\sum_{j,l}|\lambda_{j,l}|^p\frac{(1-|z_{j,l}|)^{pM-1}\omg(z_{j,l})^{-\frac{p}{q}}}{\left(1-|z_{j,l}|r\right)^{pM-1}}
  	\asymp\sum_{j=0}^\infty c_j^p\frac{(1-r_j)^{pM-1}}{\left(1-r_j r\right)^{pM-1}\omg(r_j)^{\frac{p}{q}}}.
    \end{split}
    \end{equation}

Case 1.1: ${0<q \leq p\le 1}$. Observe first that $M> \frac{\g}{q}+\frac{1}{p}$ by the hypothesis \eqref{eq:M}, and hence $qM-\frac{q}{p}>\gamma$.  By using \eqref{eq:Fpmenor1} and Lemma~\ref{Lemma:replacement-Lemmas-Memoirs}(iii) we deduce
	 \begin{equation*}
    \begin{split}
  	\|F\|_{A^{p,q}_\omega}^q
  	&\lesssim\int_0^1\left(\sum_{j=0}^\infty c_j^p\frac{(1-r_j)^{pM-1}}{\left(1-r_j r\right)^{pM-1}\omg(r_j)^{\frac{p}{q}}} \right)^\frac{q}{p}   \omega(r)\,dr\\
  	&\lesssim\sum_{j=0}^\infty c_j^q \frac{(1-r_j)^{qM-\frac{q}{p}}}{\omg(r_j)} \int_0^1 \frac{\omega(r)}{\left(1-r_j r\right)^{qM-\frac{q}{p}}}\,dr\\
  	&\asymp\sum_{j=0}^\infty c_j^q +\sum_{j}c_j^q \frac{(1-r_j)^{qM-\frac{q}{p}}}{\omg(r_j)}
    \int_0^{r_j} \frac{\omega(r)}{\left(1-r_j r\right)^{qM-\frac{q}{p}}}\,dr
    \asymp\sum_{j=0}^\infty c_j^q,
    \end{split}
    \end{equation*}
and thus \eqref{eq:Fnormcontrol} is satisfied.

Case 1.2: ${0<p\le1}$ and $p<q$. By \eqref{eq:Fpmenor1} and standard estimates
     \begin{equation*}
    \begin{split}
  	\|F\|_{A^{p,q}_\omega}^q
  	&\lesssim \int_0^1\left(\sum_{j=0}^\infty c_j^p\frac{(1-r_j)^{pM-1}}{\left(1-r_jr\right)^{pM-1}\omg(r_j)^{\frac{p}{q}}}\right)^\frac{q}{p} \omega(r)\,dr\\
  	&=\sum_{k=0}^\infty \int_{r_k}^{r_{k+1}}
    \left(\sum_{j=0}^\infty c_j^p\frac{(1-r_j)^{pM-1}}{\left(1-r_jr\right)^{pM-1}\omg(r_j)^{\frac{p}{q}}}\right)^\frac{q}{p} \omega(r)\,dr\\
  	&\asymp\sum_{k=0}^\infty \int_{r_k}^{r_{k+1}}
    \left(\sum_{j=0}^\infty c_j^p\frac{(1-r_j)^{pM-1}}{\left(1-r_jr_k\right)^{pM-1}\omg(r_j)^{\frac{p}{q}}}\right)^\frac{q}{p} \omega(r)\,dr\\
  	&\le\sum_{k=0}^\infty\left(\sum_{j=0}^\infty c_j^p\frac{(1-r_j)^{pM-1}}{\left(1-r_jr_k\right)^{pM-1}\omg(r_j)^{\frac{p}{q}}}\right)^\frac{q}{p}\omg(r_k)\\
    &\lesssim\sum_{k=0}^\infty\left(\sum_{j=k+1}^\infty c_j^p\frac{(1-r_j)^{pM-1}}{\left(1-r_jr_k\right)^{pM-1}\omg(r_j)^{\frac{p}{q}}}\right)^\frac{q}{p}\omg(r_k)\\
    &\quad+\sum_{k=0}^\infty\left(\sum_{j=0}^k c_j^p\frac{(1-r_j)^{pM-1}}{\left(1-r_jr_k\right)^{pM-1}\omg(r_j)^{\frac{p}{q}}}\right)^\frac{q}{p}\omg(r_k)\\
    &\le\sum_{k=0}^\infty\left(\sum_{j=k+1}^\infty c_j^p\frac{(1-r_j)^{pM-1}}{\left(1-r_jr_k\right)^{pM-1}
    \omg(r_j)^{\frac{p}{q}}}\right)^\frac{q}{p}\omg(r_k)\\
  	&\quad+\sum_{k=0}^\infty\left(\sum_{j=0}^k \frac{c_j^p}{\omg(r_j)^{\frac{p}{q}}}\right)^\frac{q}{p}\omg(r_k)=S_1(F)+S_2(F). 	
    \end{split}
    \end{equation*}
To prove the estimate $S_l(F)\lesssim\left\|\lambda\right\|^q_{\ell^{p,q}}$ for $l=1,2$ we will use the characterization, obtained by Muckenhoupt~\cite{Muck}, of the weights $U$ and $V$ such that the Hardy operators $\int_0^x f(t)\,dt$ and $\int_x^\infty f(t)\,dt$ are bounded from $L^s(U^s, (0,\infty))$ to $L^s(V^s, (0,\infty))$, where $s=\frac{q}{p}>1$. To do this, consider first the step functions
    \begin{equation*}
    \begin{split}
    U(x)&=\frac{\omg(r_k)^\frac{p}{q}}{(1-r_k)^{pM-1}},\quad x\in [k,k+1),\quad k\in\N\cup\{0\};\\
    f(x)&=c_k^p\frac{(1-r_k)^{pM-1}}{\omg(r_k)^{\frac{p}{q}}},\quad x\in [k,k+1),\quad k\in\N\cup\{0\};\\
    V(x)&=\frac{\omg(r_k)^\frac{p}{q}}{(1-r_k)^{pM-1}},\quad x\in [k,k+1),\quad k\in\N\cup\{0\}.
    \end{split}
    \end{equation*}
With this notation
	\begin{equation*}
    \begin{split}
 	\left\|\lambda\right\|^q_{\ell^{p,q}}=\sum_{k=0}^\infty c_k^q
 	=\int_0^\infty \left(V(x)f(x) \right)^\frac{q}{p}dx
    \end{split}
    \end{equation*}
and
    \begin{equation*}
    \begin{split}
    \int_0^\infty \left(U(x) \int_x^\infty f(y)\,dy\right)^\frac{q}{p} dx
 	\ge\sum_{k=0}^\infty\left(\sum_{j=k+1}^\infty c_j^p\frac{(1-r_j)^{pM-1}}{\left(1-r_jr_k\right)^{pM-1}
    \omg(r_j)^{\frac{p}{q}}}\right)^\frac{q}{p}\omg(r_k)=S_1(F).
    \end{split}
    \end{equation*}
Therefore the estimate $S_1(F)\lesssim\left\|\lambda\right\|^q_{\ell^{p,q}}$ follows by \cite[Theorem~2]{Muck} once it is shown that
    \begin{equation}
    \begin{split}\label{eq:H1}
    \sup_{x\ge0}\left(\int_0^xU(y)^\frac{q}{p}\,dy\right)^\frac{p}{q}
    \left(\int_x^\infty V(y)^{-\left(\frac{q}{p}\right)'}\,dy\right)^{\frac{1}{\left(\frac{q}{p}\right)'}}<\infty.
    \end{split}
    \end{equation}
To see this, let $x\ge0$, and take $N=N(x)\in\N\cup\{0\}$ such that $N\le x<N+1$.
Then Lemma~\ref{Lemma:replacement-Lemmas-Memoirs}(ii) and the inequality $M>\frac{1}{p}+\frac{\b}{q}$, which follows by the hypothesis \eqref{eq:M}, imply
    \begin{equation}
    \begin{split}\label{eq:2}
 	\left(\int_0^xU(y)^\frac{q}{p}\,dy\right)^\frac{p}{q}
    &\le\left(\sum_{k=0}^N\int_k^{k+1}U(y)^\frac{q}{p}\,dy\right)^\frac{p}{q}
    =\left(\sum_{k=0}^{N}U(k)^\frac{q}{p} \right)^\frac{p}{q}\\
 	&=\left(\sum_{k=0}^{N} \frac{\omg(r_k)}{(1-r_k)^{qM-\frac{q}{p}}}\right)^\frac{p}{q}\\
 	&\lesssim\left(\frac{\omg(r_N)}{(1-r_N)^{\b}}\sum_{k=0}^{N} \frac{1}{(1-r_k)^{qM-\frac{q}{p}-\b}}\right)^\frac{p}{q}\\
 	&\asymp \left(\frac{\omg(r_N)}{(1-r_N)^{qM-\frac{q}{p}}}\right)^\frac{p}{q}=\frac{\omg(r_N)^\frac{p}{q}}{(1-r_N)^{pM-1}}.
    \end{split}
    \end{equation}
Another application of Lemma~\ref{Lemma:replacement-Lemmas-Memoirs}(ii) and $M>\frac{1}{p}+\frac{\b}{q}$ give
    \begin{equation*}
    \begin{split}
 	\left(\int_x^\infty V(y)^{-\left(\frac{q}{p}\right)'}\,dy\right)^{\frac{1}{\left(\frac{q}{p}\right)'}}
    &\le\left(\sum_{k=N}^\infty V(k)^{-\left(\frac{q}{q-p}\right)}\right)^{\frac{q-p}{q}}
 	=\left( \sum_{k=N}^{\infty} \left(\frac{\omg(r_k)^\frac{p}{q}}{(1-r_k)^{pM-1}}\right)^{-\frac{q}{q-p}} \right)^\frac{q-p}{q}\\
 	&\lesssim\left(\frac{(1-r_N)^\frac{p\b}{q-p}}{\omg(r_N)^\frac{p}{q-p}}\sum_{k=N}^{\infty}
    \frac{1}{(1-r_k)^{\frac{p}{q-p}\left(\b-q\left(M-\frac{1}{p}\right)\right)}} \right)^\frac{q-p}{q}
 	\asymp\frac{(1-r_N)^{pM-1}}{\omg(r_N)^{\frac{p}{q}}},
    \end{split}
    \end{equation*}
which together with \eqref{eq:2} gives \eqref{eq:H1}.

We next establish $S_2(F)\lesssim\left\|\lambda\right\|^q_{\ell^{p,q}}$. Define
    \begin{equation*}
    \begin{split}
    U(x)&=\omg(r_k)^\frac{p}{q},\quad x\in [k,k+1),\quad k\in\N\cup\{0\};\\
    f(x)&=\frac{c_k^p}{\omg(r_k)^\frac{p}{q}},\quad x\in [k,k+1),\quad k\in\N\cup\{0\};\\
    V(x)&=\omg(r_k)^\frac{p}{q},\quad x\in [k,k+1),\quad k\in\N\cup\{0\}.
    \end{split}
    \end{equation*}
Then
	\begin{equation*}
    \begin{split}
 	\left\|\lambda\right\|^q_{\ell^{p,q}}
    =\sum_{k=0}^\infty c_k^q
 	=\int_0^\infty\left(V(x)f(x)\right)^\frac{q}{p}\,dx
    \end{split}
    \end{equation*}
and
    \begin{equation*}
    \begin{split}
    \int_0^\infty \left(U(x) \int_0^x f(y)\,dy\right)^\frac{q}{p} dx
 	\ge\sum_{k=0}^\infty \left(\sum_{j=0}^{k-1}\frac{c_j^p}{\omg(r_j)^{\frac{p}{q}}}\right)^\frac{q}{p}\omg(r_k),
    \end{split}
    \end{equation*}
and thus
    \begin{equation*}
    \begin{split}
    S_2(F)&\lesssim\sum_{k=0}^\infty \left(\sum_{j=0}^{k-1}\frac{c_j^p}{\omg(r_j)^{\frac{p}{q}}}\right)^\frac{q}{p}\omg(r_k)+\sum_{k=0}^\infty c_k^q
    \le\int_0^\infty \left(U(x) \int_0^x f(y)\,dy\right)^\frac{q}{p} dx+\|\lambda\|_{\ell^{p,q}}^q.
    \end{split}
    \end{equation*}
Therefore the estimate $S_2(F)\lesssim \left\|\lambda\right\|^q_{\ell^{p,q}}$ we are after, follows by \cite[Theorem~1]{Muck} if
    \begin{equation}
    \begin{split}\label{eq:H2}
 	\sup_{x\ge 0} \left(\int_x^\infty U(y)^\frac{q}{p}\,dy \right)^\frac{p}{q}
    \left(\int_0^x V(y)^{-\left(\frac{q}{p}\right)'}dy\right)^{\frac{1}{\left(\frac{q}{p}\right)'}}
    <\infty.
    \end{split}
    \end{equation}
To prove this, let $x\ge 0$ and choose $N=N(x)\in\N\cup\{0\}$ such that $N\le x<N+1$. Then
\eqref{Lemma:weights-in-R} yields
    \begin{equation*}
    \begin{split}
    \left(\int_x^\infty U(y)^\frac{q}{p}dy \right)^\frac{p}{q}
    &\le\left(\sum_{k=N}^{\infty}U(k)^\frac{q}{p}\right)^\frac{p}{q}
 	=\left(\sum_{k=N}^{\infty} \omg(r_k)\right)^\frac{p}{q}
 	\lesssim \left(\sum_{k=N}^{\infty} \left(\frac{1-r_k}{1-r_N}\right)^{\a} \omg(r_N)\right)^\frac{p}{q}
 	\asymp\omg(r_N)^{\frac{p}{q}}
    \end{split}
    \end{equation*} 	
and
	\begin{equation*}
    \begin{split}
	\left(\int_0^x V(y)^{-\left(\frac{q}{p}\right)'}dy\right)^{\frac{1}{\left(\frac{q}{p}\right)'}}
    &\le\left(\sum_{k=0}^{N} V(k)^{-\left(\frac{q}{p}\right)'}\right)^{\frac{1}{\left(\frac{q}{p}\right)'}}
 	=\left(\sum_{k=0}^{N} \frac1{\omg(r_k)^{\frac{p}{q-p}}}\right)^{\frac{1}{\left(\frac{q}{p}\right)'}}\\
 	&\lesssim\left(\sum_{k=0}^{N} \left(\frac{1-r_{N}}{1-r_k}\right)^{\frac{p\alpha}{q-p}}\frac1{\omg(r_{N})^{\frac{p}{q-p}}} \right)^{\frac{q-p}{q}}
 	\asymp\frac1{\omg(r_N)^{\frac{p}{q}}},
    \end{split}
    \end{equation*} 	
from which \eqref{eq:H2} follows. Case 1.2 is now proved.

Assume next $1<p<\infty$. Before dealing with Cases 2.1 and 2.2, we will estimate $M_p(r,F)$. We claim that there exist $\eta,\theta\in(0,1)$ such that
    \begin{equation}
    \begin{split}\label{eq:parametros}
                   & M(1-\theta)p'>1, \\
                     & M(\eta-\theta)+\frac{1-\eta}{p}>0,\\
                     & p'(1-\eta)(M-\frac{1}{p}-\frac{\b}{q})>1,\\
                     & pM\theta>1,\\
                     & q\eta (M-\frac{1}{p})>\gamma,\\
                     & M(\eta-\theta)+\frac{1-\eta}{p}<\frac{\a \eta}{q},
    \end{split}
    \end{equation}
where $\alpha$ is that of 
\eqref{Lemma:weights-in-R} and $M,\b,\g$ are those in the statement of the theorem.
We postpone the proof of this fact for a moment and estimate $M_p(r,F)$ first. By H\"{o}lder's inequality,
    \begin{equation}
    \begin{split}\label{eq:F}
  	|F(z)|^p
    &\le\sum_{j,l}|\lambda_{j,l}|^p
    \frac{(1-|z_{j,l}|)^{(pM-1)\eta}}{\left|1-\overline{z_{j,l}}z\right|^{pM\theta}\omg(z_{j,l})^{\frac{p\eta}{q}}}
    \left(\sum_{j,l}\frac{(1-|z_{j,l}|)^{p'\left(M-\frac{1}{p}\right)(1-\eta)}}{\left|1-\overline{z_{j,l}}z\right|^{p'M(1-\theta)}\omg(z_{j,l})^{\frac{p'(1-\eta)}{q}}}\right)^\frac{p}{p'}
    \end{split}
    \end{equation}
for all $z\in\D$. Since $\{z_k\}$ is separated by the hypothesis, there exists $\delta=\delta(K)>0$ such that $\Delta(z_{j,l},\delta)\subset \{r_{j-1}\le|w|<r_{j+2}\}$ for all $l$ with the convenience that $r_{-1}=r_0=0$, and $\Delta(z_{j,l_1},\delta)\cap\Delta(z_{j,l_2},\delta)=\emptyset$ if $l_1\neq l_2$. Therefore by using the subharmonicity and the first case in \eqref{eq:parametros} we obtain
    \begin{equation}\label{2}
    \begin{split}
    \sum_{l}\frac{1}{\left|1-\overline{z_{j,l}}z\right|^{p'M(1-\theta)}}
    &\lesssim K^{2j} \sum_{l} \int_{\Delta(z_{j,l},\delta)}\frac{dA(w)}{\left|1-\overline{w}z\right|^{p'M(1-\theta)}}\\
    &\le K^{2j}\int_{D(0,r_{j+2})\setminus D(0,r_{j-1})}\frac{dA(w)}{\left|1-\overline{w}z\right|^{p'M(1-\theta)}}\\
    &\lesssim   \frac{K^{j}}{\left(1-r_{j+2}|z|\right)^{p'M(1-\theta)-1}}
    \asymp \frac{ K^{j} }{\left(1-r_j|z|\right)^{p'M(1-\theta)-1}},
    \end{split}
    \end{equation}
and hence
    \begin{equation}
    \begin{split}\label{eq:3}
    \sum_{j,l}\frac{(1-|z_{j,l}|)^{p'\left(M-\frac{1}{p}\right)(1-\eta)}}{\left|1-\overline{z_{j,l}}z\right|^{p'M(1-\theta)}\omg(z_{j,l})^{\frac{p'(1-\eta)}{q}}}
    &\lesssim\sum_j\frac{(1-r_j)^{p'\left(M-\frac{1}{p}\right)(1-\eta)-1}}{\left(1-r_j|z|\right)^{p'M(1-\theta)-1}\omg(r_j)^{\frac{p'(1-\eta)}{q}}}\\
    &\asymp\sum_{r_j\le|z|}\frac{(1-r_j)^{p'\left(M-\frac{1}{p}\right)(1-\eta)-1}}{\left(1-r_j|z|\right)^{p'M(1-\theta)-1}\omg(r_j)^{\frac{p'(1-\eta)}{q}}}\\
    &\quad+\sum_{r_j>|z|}\frac{(1-r_j)^{p'\left(M-\frac{1}{p}\right)(1-\eta)-1}}{\left(1-r_j|z|\right)^{p'M(1-\theta)-1}\omg(r_j)^{\frac{p'(1-\eta)}{q}}}\\
    &\le\frac1{\omg(z)^{\frac{p'(1-\eta)}{q}}}\sum_{r_j\leq |z|}(1-r_j)^{p'\left(M-\frac{1}{p}\right)(1-\eta)-p'M(1-\theta)}\\
    &\quad+\frac{1}{\left(1-|z|\right)^{p'M(1-\theta)-1}}\sum_{r_j>|z|}\frac{(1-r_j)^{p'\left(M-\frac{1}{p}\right)(1-\eta)-1}}{\omg(r_j)^{\frac{p'(1-\eta)}{q}}}\\
    &=S_3(F)+S_4(F).
    \end{split}
    \end{equation}
Since $M(\eta-\theta)+\frac{1-\eta}{p}>0$ by the second case in \eqref{eq:parametros},
    \begin{equation}
    \begin{split}\label{eq:4}
    S_3(F)&
    \le\frac1{\omg(z)^{\frac{p'(1-\eta)}{q}}}\sum_{r_j\le|z|}\frac1{(1-r_j)^{p'\left(M(\eta-\theta)+\frac{1-\eta}{p}\right)}}
    \asymp\frac{1}{\omg(z)^{\frac{p'(1-\eta)}{q}}(1-|z|)^{p'\left(M(\eta-\theta)+\frac{1-\eta}{p}\right)}}.
    \end{split}
    \end{equation}
Now, by using Lemma~\ref{Lemma:replacement-Lemmas-Memoirs}(ii) and the third case of \eqref{eq:parametros}, we deduce
    \begin{equation}
    \begin{split}\label{eq:5}
    S_4(F)
    &\lesssim \frac{(1-|z|)^{\frac{\b p'(1-\eta)}{q}-(p'M(1-\theta)-1)}}{\omg(z)^{\frac{p'(1-\eta)}{q}}} \sum_{r_j> |z|}(1-r_j)^{p'\left(M-\frac{1}{p}\right)(1-\eta)-1-\frac{\b p'(1-\eta)}{q}}\\
    &\asymp\frac{1}{\omg(z)^{\frac{p'(1-\eta)}{q}}(1-|z|)^{p'\left(M(\eta-\theta)+\frac{1-\eta}{p}\right)}}.
    \end{split}
    \end{equation}
Consequently, by combining the estimates \eqref{eq:F}--\eqref{eq:5} we deduce
    \begin{equation*}
    \begin{split}
  	|F(z)|^p
  	&\lesssim\frac{\omg(z)^{-\frac{p(1-\eta)}{q}}}{(1-|z|)^{p\left(M(\eta-\theta)+\frac{1-\eta}{p}\right)}}
  	\sum_{j,l}|\lambda_{j,l}|^p\frac{(1-|z_{j,l}|)^{(pM-1)\eta}}{\left|1-\overline{z_{j,l}}z\right|^{pM\theta}\omg(z_{j,l})^{\frac{p\eta}{q}}},\quad z\in\D,
    \end{split}
    \end{equation*}
and hence the fourth case of \eqref{eq:parametros} gives
    \begin{equation}
    \begin{split}\label{eq:MprF}
  	M^p_p(r,F)
  	&\lesssim\frac{\omg(r)^{-\frac{p(1-\eta)}{q}}}{(1-r)^{p\left(M(\eta-\theta)+\frac{1-\eta}{p}\right)}}
    \sum_{j,l}|\lambda_{j,l}|^p\frac{(1-r_j)^{(pM-1)\eta}}{\omg(r_j)^{\frac{p\eta}{q}}}
    \int_0^{2\pi}\frac{dt}{\left|1-\overline{z_{j,l}}re^{it}\right|^{pM\theta}}\\
  	&\lesssim\frac{\omg(r)^{-\frac{p(1-\eta)}{q}}}{(1-r)^{p\left(M(\eta-\theta)+\frac{1-\eta}{p}\right)}}
  	\sum_{j}c_j^p\frac{(1-r_j)^{(pM-1)\eta}}{\left(1-r_j r\right)^{pM\theta-1}\omg(r_j)^{\frac{p\eta}{q}}}.
    \end{split}
    \end{equation}
By using this estimate we will deal with Cases 2.1 and 2.2.

Case 2.1: $p>1$ and $0<q \le p$. By \eqref{eq:MprF},	
	\begin{equation}
    \begin{split}\label{eq:6}
  	&\|F\|_{A^{p,q}_\omega}^q
  	\lesssim \int_0^1 \frac{\omg(r)^{-(1-\eta)}}{(1-r)^{q\left(M(\eta-\theta)+\frac{1-\eta}{p}\right)}}
  	\sum_{j}c_j^q\frac{(1-r_j)^{q\left(M-\frac{1}{p}\right)\eta}\omg(r_j)^{-\eta}}{\left(1-r_j r\right)^{qM\theta-\frac{q}{p}}}\omega(r)\,dr \\
  	&= \sum_{j}c_j^q (1-r_j)^{q\left(M-\frac{1}{p}\right)\eta}\omg(r_j)^{-\eta}
  	\int_0^1\frac{\omg(r)^{-(1-\eta)}}{\left(1-r_j r\right)^{qM\theta-\frac{q}{p}}(1-r)^{q\left(M(\eta-\theta)+\frac{1-\eta}{p}\right)}}\omega(r)\,dr.
    \end{split}
    \end{equation}
Lemma~\ref{Lemma:replacement-Lemmas-Memoirs}(iii) together with the fifth case of \eqref{eq:parametros} yields
    \begin{equation}
    \begin{split}\label{eq:7}
  	\int_0^{r_j}\frac{\omg(r)^{-(1-\eta)}\omega(r)}{\left(1-r_j r\right)^{qM\theta-\frac{q}{p}}(1-r)^{q\left(M(\eta-\theta)+\frac{1-\eta}{p}\right)}}\,dr
  	&\le\int_0^{r_j}\frac{\omg(r)^{-(1-\eta)}\omega(r)}{(1-r)^{q\left(M(\eta-\theta)+\frac{1-\eta}{p}\right)+qM\theta-\frac{q}{p}}}\,dr\\
    &\le\frac{1}{\widehat{\om}(r_j)^{1-\eta}}\int_0^{r_j}\frac{\omega(r)}{(1-r)^{q\eta\left(M-\frac{1}{p}\right)}}\,dr\\
  	&\lesssim\frac{\omg(r_j)^\eta}{\left(1-r_j \right)^{q\eta\left(M-\frac{1}{p}\right)}},
    \end{split}
    \end{equation}
while the fourth case of \eqref{eq:parametros} implies
     \begin{equation*}
    \begin{split}
  	\int_{r_j}^1\frac{\omg(r)^{-(1-\eta)}\omega(r)}{\left(1-r_j r\right)^{qM\theta-\frac{q}{p}}(1-r)^{q\left(M(\eta-\theta)+\frac{1-\eta}{p}\right)}}\,dr
  	\le\frac{1}{(1-r_j)^{Mq\theta -\frac{q}{p}}} \int_{r_j}^1 \frac{\omg(r)^{-(1-\eta)}\omega(r)}{(1-r)^{q\left(M(\eta-\theta)+\frac{1-\eta}{p}\right)}}\,dr,
    \end{split}
    \end{equation*}
where, by an integration by parts and 
\eqref{Lemma:weights-in-R} together with the sixth case of \eqref{eq:parametros},
    \begin{equation*}
    \begin{split}
    &\int_{r_j}^1 \frac{\omg(r)^{-(1-\eta)}}{(1-r)^{q\left(M(\eta-\theta)+\frac{1-\eta}{p}\right)}}\omega(r)\,dr\\
    &=\frac{\omg(r_j)^\eta}{\eta(1-r_j)^{q\left(M(\eta-\theta)+\frac{1-\eta}{p}\right)}}
    +\frac{1}{\eta q(M(\eta-\theta)+\frac{1-\eta}{p})}
    \int_{r_j}^1\frac{\omg(r)^{\eta}}{(1-r)^{q\left(M(\eta-\theta)+\frac{1-\eta}{p}\right)+1}}\,dr\\
    &\lesssim\frac{\omg(r_j)^\eta}{(1-r_j)^{q\left(M(\eta-\theta)+\frac{1-\eta}{p}\right)}}+\frac{\omg(r_j)^\eta}{(1-r_j)^{\a\eta}}
    \int_{r_j}^1\frac{dr}{(1-r)^{q\left(M(\eta-\theta)+\frac{1-\eta}{p}\right)+1-\a\eta}}\\
    &\asymp\frac{\omg(r_j)^\eta}{(1-r_j)^{q\left(M(\eta-\theta)+\frac{1-\eta}{p}\right)}},
    \end{split}
    \end{equation*}
and thus
     \begin{equation*}
    \begin{split}
  	&\int_{r_j}^1\frac{\omg(r)^{-(1-\eta)}\omega(r)}
    {\left(1-r_j r\right)^{qM\theta-\frac{q}{p}}(1-r)^{q\left(M(\eta-\theta)+\frac{1-\eta}{p}\right)}}\,dr
  	\lesssim\frac{\omg(r_j)^\eta}{\left(1-r_j \right)^{q\eta\left(M-\frac{1}{p}\right)}}.
    \end{split}
    \end{equation*}
This together with \eqref{eq:6} and \eqref{eq:7} gives \eqref{eq:Fnormcontrol}.

Case 2.2: $p>1$ and $0<p<q$. By \eqref{eq:MprF} and the fourth case of \eqref{eq:parametros},
     \begin{equation*}
    \begin{split}
  	&\|F\|_{A^{p,q}_\omega}^q
  	\lesssim\int_0^1\left(\frac{\omg(r)^{-\frac{p(1-\eta)}{q}}}{(1-r)^{p\left(M(\eta-\theta)+\frac{1-\eta}{p}\right)}}
  	\sum_{j=0}^\infty c_j^p\frac{(1-r_j)^{(pM-1)\eta}\omg(r_j)^{-\frac{p\eta}{q}}}{\left(1-r_j r\right)^{pM\theta-1}} \right)^\frac{q}{p}\omega(r)\,dr\\
  	&=\sum_{k=0}^\infty\int_{r_k}^{r_{k+1}} \left(\frac{\omg(r)^{-\frac{p(1-\eta)}{q}}}{(1-r)^{p\left(M(\eta-\theta)+\frac{1-\eta}{p}\right)}}
  	\sum_{j=0}^\infty c_j^p\frac{(1-r_j)^{(pM-1)\eta}\omg(r_j)^{-\frac{p\eta}{q}}}{\left(1-r_j r\right)^{pM\theta-1}}\right)^\frac{q}{p}\omega(r)\,dr\\
  	&\lesssim\sum_{k=0}^\infty\left(\frac{\omg(r_k)^{-\frac{p(1-\eta)}{q}}}{(1-r_k)^{p\left(M(\eta-\theta)+\frac{1-\eta}{p}\right)}}
  	\sum_{j=0}^\infty c_j^p\frac{(1-r_j)^{(pM-1)\eta}\omg(r_j)^{-\frac{p\eta}{q}}}
    {\left(1-r_j r_k\right)^{pM\theta-1}}\right)^\frac{q}{p}(\omg(r_k)-\omg(r_{k+1}))\\
    &\le\sum_{k=0}^\infty\left(\frac{\omg(r_k)^{\frac{p\eta}{q}}}{(1-r_k)^{p\left(M(\eta-\theta)+\frac{1-\eta}{p}\right)}}
  	\sum_{j=0}^\infty c_j^p\frac{(1-r_j)^{(pM-1)\eta}}{\left(1-r_j r_k\right)^{pM\theta-1}\omg(r_j)^{\frac{p\eta}{q}}}\right)^\frac{q}{p}\\
  	&\lesssim\sum_{k=0}^\infty\left(\frac{\omg(r_k)^{\frac{p\eta}{q}}}{(1-r_k)^{p\left(M(\eta-\theta)+\frac{1-\eta}{p}\right)}}
  	\sum_{j=k+1}^\infty c_j^p\frac{(1-r_j)^{(pM-1)\eta}}{\left(1-r_j r_k\right)^{pM\theta-1}\omg(r_j)^{\frac{p\eta}{q}}}\right)^\frac{q}{p}\\
  	&\quad+ \sum_{k=0}^\infty\left(\frac{\omg(r_k)^{\frac{p\eta}{q}}}{(1-r_k)^{p\left(M(\eta-\theta)+\frac{1-\eta}{p}\right)}}
  	\sum_{j=0}^k c_j^p\frac{(1-r_j)^{(pM-1)\eta}}{\left(1-r_j r_k\right)^{pM\theta-1}\omg(r_j)^{\frac{p\eta}{q}}}\right)^\frac{q}{p}\\
    &\lesssim \sum_{k=0}^\infty \left(\frac{\omg(r_k)^{\frac{p\eta}{q}}}{(1-r_k)^{\eta(pM-1)}}
    \sum_{j=k+1}^\infty c_j^p\frac{(1-r_j)^{(pM-1)\eta}}{\omg(r_j)^{\frac{p\eta}{q}}}\right)^\frac{q}{p} \\
  	&\quad+\sum_{k=0}^\infty\left(\frac{\omg(r_k)^{\frac{p\eta}{q}}}{(1-r_k)^{p\left(M(\eta-\theta)+\frac{1-\eta}{p}\right)}}
  	\sum_{j=0}^k c_j^p\frac{(1-r_j)^{pM(\eta-\theta)+(1-\eta)}}{\omg(r_j)^{\frac{p\eta}{q}}}\right)^\frac{q}{p}
    =S_5(F)+S_6(F).
    \end{split}
    \end{equation*}
To prove the estimate $S_5(F)\lesssim\|\lambda\|^q_{\ell^{p,q}}$, define the step functions
    \begin{equation*}
    \begin{split}
    U(x)&=\frac{\omg(r_k)^\frac{p\eta}{q}}{(1-r_k)^{\eta(pM-1)}},\quad x\in [k,k+1),\quad k\in\N\cup\{0\};\\
    f(x)&=c_k^p\frac{(1-r_k)^{\eta(pM-1)}}{\omg(r_k)^{\frac{p\eta}{q}}},\quad x\in [k,k+1),\quad k\in\N\cup\{0\};\\
    V(x)&=\frac{\omg(r_k)^\frac{p\eta}{q}}{(1-r_k)^{\eta(pM-1)}},\quad x\in [k,k+1),\quad k\in\N\cup\{0\}.
    \end{split}
    \end{equation*}
Then
	\begin{equation*}
    \begin{split}
 	\left\|\lambda\right\|^q_{\ell^{p,q}}=\sum_k^\infty c_k^q
 	&=\int_0^\infty\left(V(x)f(x)\right)^\frac{q}{p}dx
    \end{split}
    \end{equation*}
and
    \begin{equation*}
    \begin{split}
    \int_0^\infty \left(U(x)\int_x^\infty f(y)\,dy\right)^\frac{q}{p} dx
    \ge \sum_{k=0}^\infty\left(\frac{\omg(r_k)^{\frac{p\eta}{q}}}{(1-r_k)^{\eta(pM-1)}}
    \sum_{j=k+1}^\infty c_j^p\frac{(1-r_j)^{\eta(pM-1)}}{\omg(r_j)^{\frac{p\eta}{q}}}\right)^\frac{q}{p}=S_5(F).
    \end{split}
    \end{equation*}
Therefore $S_5(F)\lesssim \left\|\lambda\right\|^q_{\ell^{p,q}}$ follows by \cite[Theorem~2]{Muck} if
    \begin{equation}
    \begin{split}\label{eq:H3}
 	\sup_{x\ge0}\left(\int_0^x U(y)^\frac{q}{p}dy \right)^\frac{p}{q}  \left(\int_x^\infty
 V(y)^{-\left(\frac{q}{p}\right)'}dy\right)^{\frac{1}{\left(\frac{q}{p}\right)'}}  <\infty.
    \end{split}
    \end{equation}
To prove this, let $x\ge 0$ and take $N=N(x)\in\N\cup\{0\}$ such that $N\le x<N+1$.
Then Lemma~\ref{Lemma:replacement-Lemmas-Memoirs}(ii) and the hypothesis \eqref{eq:M} yield
    \begin{equation}
    \begin{split}\label{eq:8}
 	\left(\int_0^xU(y)^\frac{q}{p}\,dy\right)^\frac{p}{q}
    &\le\left( \sum_{k=0}^{N+1} |U(k)|^\frac{q}{p} \right)^\frac{p}{q}
 	=\left(\sum_{k=0}^{N} \left(\frac{\omg(r_k)^\frac{p\eta}{q}}{(1-r_k)^{\eta(pM-1)}}\right)^\frac{q}{p} \right)^\frac{p}{q}\\
 	&=\left(\sum_{k=0}^{N}\frac{\omg(r_k)^\eta}{(1-r_k)^{\eta\left(qM-\frac{q}{p}\right)}}\right)^\frac{p}{q}
 	\lesssim \left(\frac{\omg(r_N)^\eta}{(1-r_N)^{\b\eta}}\sum_{k=0}^{N}\frac{1}{(1-r_k)^{\eta\left(qM-\frac{q}{p}-\beta\right)}}\right)^\frac{p}{q}\\
 	&\asymp \left(\frac{\omg(r_N)^\eta}{(1-r_N)^{\eta\left(qM-\frac{q}{p}\right)}}\right)^\frac{p}{q}
    =\frac{\omg(r_N)^\frac{p\eta}{q}}{(1-r_N)^{\eta(pM-1)}}.
    \end{split}
    \end{equation}
Another application of Lemma~\ref{Lemma:replacement-Lemmas-Memoirs}(ii) and the hypothesis \eqref{eq:M} give
    \begin{equation*}
    \begin{split}
 	\left(\int_x^\infty V(y)^{-\left(\frac{q}{p}\right)'}dy\right)^{\frac{1}{\left(\frac{q}{p}\right)'}}
 	&\le\left( \sum_{k=N}^{\infty} \left(\frac{\omg(r_k)^\frac{p\eta}{q}}{(1-r_k)^{\eta(pM-1)}}\right)^{-\frac{q}{q-p}} \right)^\frac{q-p}{q}\\
 	&\lesssim \left(\frac{(1-r_N)^\frac{p\b\eta}{q-p}}{\omg(r_N)^\frac{p\eta}{q-p}}\sum_{k=N}^{\infty} \frac{1}{(1-r_k)^{\frac{p\eta}{q-p}(\b-q(M-\frac{1}{p}))}} \right)^\frac{q-p}{q}\\
 	& \asymp \frac{(1-r_N)^{\eta(pM-1)}}{\omg(r_N)^\frac{p\eta}{q}},
    \end{split}
    \end{equation*}
which together with \eqref{eq:8} implies \eqref{eq:H3}.

We next prove $S_6(F)\lesssim \left\|\lambda\right\|^q_{\ell^{p,q}}$. Define
    \begin{equation*}
    \begin{split}
    U(x)&=\frac{\omg(r_k)^{\frac{p\eta}{q}}}{(1-r_k)^{p(M(\eta-\theta)+\frac{1-\eta}{p})}},\quad x\in [k,k+1),\quad k\in\N\cup\{0\};\\ f(x)&=c_k^p\frac{(1-r_k)^{pM(\eta-\theta) +1-\eta}}{\omg(r_k)^{\frac{p\eta}{q}}},\quad x\in [k,k+1),\quad k\in\N\cup\{0\};\\
    V(x)&=\frac{\omg(r_k)^\frac{p\eta}{q}}{(1-r_k)^{pM(\eta-\theta)+1-\eta}},\quad x\in [k,k+1),\quad k\in\N\cup\{0\}.
    \end{split}
    \end{equation*}
Then
	\begin{equation*}
    \begin{split}
 	\left\|\lambda\right\|^q_{\ell^{p,q}}=\sum_{k=0}^\infty c_k^q
 	=\int_0^\infty \left(V(x)f(x) \right)^\frac{q}{p}dx
    \end{split}
    \end{equation*}
and
    \begin{equation*}
    \begin{split}
    \int_0^\infty \left(U(x) \int_0^x f(y)\,dy\right)^\frac{q}{p} dx
    &\ge\sum_{k=0}^\infty\left(U(k)\sum_{j=0}^{k-1}f(j)\right)^\frac{q}{p}\\
    &=\sum_{k=0}^\infty\left(\frac{\omg(r_k)^{\frac{p\eta}{q}}}{(1-r_k)^{p\left(M(\eta-\theta)+\frac{1-\eta}{p}\right)}}
  	\sum_{j=0}^{k-1}c_j^p\frac{(1-r_j)^{pM(\eta-\theta)+(1-\eta)}}{\omg(r_j)^{\frac{p\eta}{q}}}\right)^\frac{q}{p},
    \end{split}
    \end{equation*}
and hence
    $$
    S_6(F)\lesssim\int_0^\infty \left(U(x) \int_0^x f(y)\,dy\right)^\frac{q}{p} dx+\left\|\lambda\right\|^q_{\ell^{p,q}}.
    $$
Therefore $S_6(F)\lesssim \left\|\lambda\right\|^q_{\ell^{p,q}}$ follows by \cite[Theorem~1]{Muck} once we have shown that
    \begin{equation}
    \begin{split}\label{eq:H4}
 	\sup_{x\ge 0} \left(\int_x^\infty U(y)^\frac{q}{p}dy \right)^\frac{p}{q}
    \left(\int_0^x V(y)^{-\left(\frac{q}{p}\right)'}dy\right)^{\frac{1}{\left(\frac{q}{p}\right)'}}<\infty.
    \end{split}
    \end{equation}
To see this, let $x\ge 0$ and choose $N=N(x)\in\N\cup\{0\}$ such that $N\le x<N+1$.
Then, by 
\eqref{Lemma:weights-in-R} and the sixth case of \eqref{eq:parametros} we deduce
    \begin{equation*}
    \begin{split}
 	\left(\int_x^\infty U(y)^\frac{q}{p}\,dy\right)^\frac{p}{q}
 	&\le\left(\sum_{k=N}^{\infty} \frac{\omg(r_k)^{\eta}}{(1-r_k)^{q\left(M(\eta-\theta)+\frac{1-\eta}{p}\right)}}\right)^\frac{p}{q}\\
    &\lesssim\left(\frac{\omg(r_N)^{\eta}}{(1-r_N)^{\eta\a}}\sum_{k=N}^{\infty} \frac{1}{(1-r_k)^{q\left(M(\eta-\theta)+\frac{1-\eta}{p}\right)-\eta\a}} \right)^\frac{p}{q}\\
 	&\asymp \frac{\omg(r_N)^{\frac{p\eta}{q}}}{(1-r_N)^{p\left(M(\eta-\theta)+\frac{1-\eta}{p}\right)}}
    \end{split}
    \end{equation*} 	
and
    \begin{equation*}
    \begin{split}
	\left(\int_0^x V(y)^{-\left(\frac{q}{p}\right)'}dy\right)^{\frac{1}{\left(\frac{q}{p}\right)'}}
 	&\le\left(\sum_{k=0}^{N}\left(\frac{(1-r_k)^{pM(\eta-\theta) +1-\eta}}{\omg(r_k)^{\frac{p\eta}{q}} }\right)^{\frac{q}{q-p}}\right)^\frac{q-p}{q}\\
    &\lesssim \left(\frac{(1-r_{N})^{\frac{p\a\eta}{q-p}}}{\omg(r_{N})^\frac{p\eta}{q-p}}
    \sum_{k=0}^{N}(1-r_k)^{\frac{q}{q-p}\left(pM(\eta-\theta)+1-\eta-\frac{p\eta\a}{q}\right)}\right)^{\frac{q-p}{q}}\\
    &\lesssim\frac{(1-r_{N})^{p\left(M(\eta-\theta)+\frac{1-\eta}{p}\right)}}{\omg(r_{N})^{\frac{p\eta}{q}}},
    \end{split}
    \end{equation*}
from which \eqref{eq:H4} follows. This finishes the proof of Case~2.2.

Finally, let us prove that there exist $\eta,\theta$ satisfying \eqref{eq:parametros}. By \eqref{eq:M}, the third and fifth cases in \eqref{eq:parametros} are equivalent to
    \begin{equation*}
    \eta\in\left(\frac{\gamma}{q\left(M-\frac{1}{p}\right)},1-\frac{1}{p'\left(M-\frac{1}{p}-\frac{\b}{q} \right)}\right),
    \end{equation*}
where the inequality $\frac{\gamma}{q\left(M-\frac{1}{p}\right)}
<1-\frac{1}{p'\left(M-\frac{1}{p}-\frac{\b}{q} \right) }$  follows from \eqref{eq:M}.
Let us observe that \eqref{eq:M} also implies $\frac{\gamma}{q\left(M-\frac{1}{p}\right)}<\frac{M-1}{M-\frac{1}{p}}$. Therefore we may choose an $\eta$ satisfying
    \begin{equation}\label{eq:par1}
    \eta\in\left(\frac{\gamma}{q\left(M-\frac{1}{p}\right)},\min\left\{1-\frac{1}{p'\left(M-\frac{1}{p}-\frac{\b}{q} \right)},\frac{M-1}{M-\frac{1}{p}} \right\}\right).
    \end{equation}
Next, observe that the first and the fourth cases in \eqref{eq:parametros} are equivalent to
\begin{equation*}
\theta\in\left( \frac{1}{Mp},1-\frac{1}{Mp'}\right),
\end{equation*}
where $\frac{1}{Mp}<1-\frac{1}{Mp'}$ by \eqref{eq:M}. Further, by \eqref{eq:M}, the second and the sixth conditions in  \eqref{eq:parametros} are equivalent to
\begin{equation}\label{eq:par3}
\theta\in\left( \frac{\left(M-\frac{1}{p}-\frac{\a}{q} \right)\eta+\frac{1}{p}}{M},\frac{\left(M-\frac{1}{p} \right)\eta+\frac{1}{p}}{M}\right),
\end{equation}
where trivially $\frac{\left(M-\frac{1}{p}-\frac{\a}{q} \right)\eta+\frac{1}{p}}{M}<\frac{\left(M-\frac{1}{p} \right)\eta+\frac{1}{p}}{M}$.
It is clear that $ \frac{1}{Mp}<\frac{\left(M-\frac{1}{p}-\frac{\a}{q} \right)\eta+\frac{1}{p}}{M}$ as $\alpha\le \beta$, and
$\frac{\left(M-\frac{1}{p} \right)\eta+\frac{1}{p}}{M}<1-\frac{1}{Mp'}$ by \eqref{eq:par1}. Therefore it is enough to choose $\theta$ satisfying \eqref{eq:par3}. This finishes the proof of \eqref{eq:parametros}.

Case 3.1: $p=\infty$ and $0<q\le1$. For this we set $\lambda(j)=\sup_l|\lambda_{j,l}|$. Then the estimates in \eqref{2} imply
    \begin{equation}\label{3}
    \begin{split}
  	|F(z)|\lesssim\sum_{j=0}^\infty\lambda(j)\frac{(1-r_j)^{M-1}}{\left(1-r_j|z|\right)^{M-1}\omg(r_j)^{\frac{1}{q}}},\quad z\in\D,
    \end{split}
    \end{equation}
for each $M>1$. Since $0<q\le1$ and $q(M-1)>\gamma$ by the hypothesis \eqref{eq:M}, Lemma~\ref{Lemma:replacement-Lemmas-Memoirs}(iii) yields
	\begin{equation*}
	 \begin{split}
  	\|F\|_{A^{\infty,q}_\om}^q
  	\lesssim \sum_{j=0}^\infty\lambda(j)^q\int_0^1\frac{(1-r_j)^{q(M-1)}}{\left(1-r_jr\right)^{q(M-1)}\omg(r_j)}\omega(r)\,dr
  	\lesssim \sum_{j=0}^\infty\lambda(j)^q=\|\lambda\|_{\ell^{\infty,q}}^q,
    \end{split}
    \end{equation*}
and thus this case is proved.

Case 3.2: $p=\infty$ and $1<q<\infty$. By using \eqref{3} we deduce
	\begin{equation*}
	\begin{split}
  	\|F\|_{A^{\infty,q}_\om}^q
    &\lesssim\int_0^1\left(\sum_{j=0}^\infty\lambda(j)
    \frac{(1-|z_{j,l}|)^{M-1}}{\left(1-r_jr\right)^{M-1}\omg(z_{j,l})^{\frac{1}{q}}}\right)^q\omega(r)\,dr\\
    &\lesssim\sum_{k=0}^\infty\left(\sum_{j=0}^\infty\lambda(j)
    \left(\frac{1-r_j}{1-r_jr_k}\right)^{M-1}\left(\frac{\omg(r_k)}{\omg(r_j)}\right)^{\frac{1}{q}}\right)^q\\
    &\lesssim\sum_{k=0}^\infty\left(\sum_{j=k+1}^\infty\lambda(j)
    \left(\frac{1-r_j}{1-r_k}\right)^{M-1}\left(\frac{\omg(r_k)}{\omg(r_j)}\right)^{\frac{1}{q}}\right)^q\\
    &\quad+\sum_{k=0}^\infty\left(\sum_{j=0}^k\lambda(j)\left(\frac{\omg(r_k)}{\omg(r_j)}\right)^{\frac{1}{q}}\right)^q
    =S_7(F) + S_8(F).
    \end{split}
    \end{equation*}
To prove $S_7(F)\lesssim \left\|\lambda\right\|^q_{\ell^{\infty,q}}$, define the step functions
    \begin{equation*}
    \begin{split}
    U(x) &= \frac{\omg(r_k)^\frac{1}{q}}{(1-r_k)^{M-1}},\quad x\in [k,k+1),\quad k\in\N\cup\{0\};\\
    f(x) & =\lambda(k)\frac{(1-r_k)^{M-1}}{\omg(r_k)^\frac{1}{q}},\quad x\in [k,k+1),\quad k\in\N\cup\{0\};\\
    V(x)& = \frac{\omg(r_k)^\frac{1}{q}}{(1-r_k)^{M-1}},\quad x\in [k,k+1),\quad k\in\N\cup\{0\}.
    \end{split}
    \end{equation*}
Then
	\begin{equation*}
    \begin{split}
 	\left\|\lambda\right\|^q_{\ell^{\infty,q}}=\sum_{k=0}^\infty\lambda(k)^q
 	=\int_0^\infty \left(V(x)f(x) \right)^q\, dx
    \end{split}
    \end{equation*}
and
    \begin{equation*}
    \begin{split}
    \int_0^\infty \left(U(x)\int_x^\infty f(y)\,dy\right)^q\,dx
 	\ge\sum_{k=0}^\infty\left(\sum_{j=k+1}^\infty\lambda(j)
    \left(\frac{1-r_j}{1-r_k}\right)^{M-1}\left(\frac{\omg(r_k)}{\omg(r_j)}\right)^{\frac{1}{q}}\right)^q =S_7(F).
    \end{split}
    \end{equation*}
Therefore $S_7(F)\lesssim \left\|\lambda\right\|^q_{\ell^{\infty,q}}$ follows by \cite[Theorem~2]{Muck}, if
    \begin{equation}
    \begin{split}\label{eq:H12}
 	\sup_{x\ge 0} \left(\int_0^x U(y)^q\,dy \right)^\frac{1}{q}
    \left(\int_x^\infty V(y)^{-q'}\,dy\right)^{\frac{1}{q'}}  <\infty.
    \end{split}
    \end{equation}
To prove this, let $x\ge 0$ and $N\in\N\cup\{0\}$ such that $N\le x<N+1$.
Then Lemma~\ref{Lemma:replacement-Lemmas-Memoirs}(ii) and the hypothesis \eqref{eq:M} imply
    \begin{equation*}
    \begin{split}
 	\int_0^x U(y)^q\,dy
 	&\le\sum_{k=0}^{N}\frac{\omg(r_k)}{(1-r_k)^{q(M-1)}}\\
 	&\lesssim\frac{\omg(r_{N})}{(1-r_{N})^{\b}}\sum_{k=0}^{N} \frac{1}{(1-r_k)^{q(M-1)-\b}}
 	\asymp\frac{\omg(r_{N})}{(1-r_{N})^{q(M-1)}}
    \end{split}
    \end{equation*}
and
    \begin{equation*}
    \begin{split}
 	\int_x^\infty V(y)^{-q'}\,dy
    &\le\sum_{k=N}^{\infty}\frac{(1-r_k)^{q'(M-1)}}{\omg(r_k)^{\frac{q'}{q}}}
 	\lesssim\frac{(1-r_N)^{\frac{q'\b}{q}}}{\omg(r_N)^{\frac{q'}{q}}}\sum_{k=N}^{\infty} (1-r_k)^{q'\left(M-1-\frac{\b}{q}\right)}\\
 	&\asymp\frac{(1-r_N)^{q'(M-1)}}{\omg(r_N)^\frac{q'}{q}},
    \end{split}
    \end{equation*}
from which \eqref{eq:H12} follows. Thus $S_7(F)\lesssim \left\|\lambda\right\|^q_{\ell^{\infty,q}}$.

To obtain $S_8(F)\lesssim \left\|\lambda\right\|^q_{\ell^{\infty,q}}$, define
    \begin{equation*}
    \begin{split}
    U(x)&=\omg(r_k)^\frac{1}{q},\quad x\in [k,k+1),\quad k\in\N\cup\{0\};\\
    f(x)&=\frac{\lambda(k)}{\omg(r_k)^\frac{1}{q}},\quad x\in[k,k+1),\quad k\in\N\cup\{0\};\\
    V(x)&=\omg(r_k)^\frac{1}{q},\quad x\in [k,k+1),\quad k\in\N\cup\{0\}.
    \end{split}
    \end{equation*}
Then
	\begin{equation*}
    \begin{split}
 	\left\|\lambda\right\|^q_{\ell^{\infty,q}}
    =\sum_{k=0}^\infty\lambda(k)^q=\int_0^\infty\left(V(x)f(x)\right)^q\,dx
    \end{split}
    \end{equation*}
and
    \begin{equation*}
    \begin{split}
    \int_0^\infty \left(U(x) \int_0^x f(y)\,dy\right)^q\,dx
 	\ge\sum_{k=0}^\infty\left(\sum_{j=0}^{k-1}\lambda(j)\left(\frac{\omg(r_k)}{\omg(r_j)}\right)^{\frac{1}{q}}\right)^q,
    \end{split}
    \end{equation*}
and hence
    $$
    S_8(F)\lesssim\int_0^\infty \left(U(x) \int_0^x f(y)\,dy\right)^q\,dx+\|\lambda\|_{\ell^{\infty,q}}^q.
    $$
Therefore $S_8(F)\lesssim\left\|\lambda\right\|^q_{\ell^{\infty,q}}$ holds by \cite[Theorem~1]{Muck}, once we have shown that
    \begin{equation}
    \begin{split}\label{eq:H10}
 	\sup_{x\ge 0}\left(\int_x^\infty U(y)^q\,dy \right)^\frac{1}{q}
    \left(\int_0^xV(y)^{-q'}\,dy\right)^{\frac{1}{q'}}<\infty.
    \end{split}
    \end{equation}
To see this, let $x\ge 0$ and $N\in\N\cup\{0\}$ such that $N\le x<N+1$.
By using 
\eqref{Lemma:weights-in-R} we deduce
    \begin{equation*}
    \begin{split}
 	\int_x^\infty U(y)^q\,dy
    &\le\sum_{k=N}^{\infty}\omg(r_k)
 	\lesssim\frac{\omg(r_N)}{(1-r_N)^{\alpha}}\sum_{k=N}^{\infty}(1-r_k)^{\alpha}
    \asymp\omg(r_N)
    \end{split}
    \end{equation*}
and
    \begin{equation*}
    \begin{split}
 	\int_0^xV(y)^{-q'}\,dy
    =\sum_{k=0}^{N}\omg(r_k)^{\frac{-q'}{q}}
 	\lesssim\frac{(1-r_{N})^{\frac{\alpha q'}{q}}}{\omg(r_{N})^{\frac{q'}{q}}}
    \sum_{k=0}^{N}(1-r_k)^{-\alpha\frac{q'}{q}}
 	\asymp\omg(r_{N})^{-\frac{q'}{q}},
    \end{split}
    \end{equation*}
from which \eqref{eq:H10} follows. Thus $S_8(F)\lesssim\left\|\lambda\right\|^q_{\ell^{\infty,q}}$. This finishes the proof of Case 3.2 and the proof of the theorem as well.

\section{A representation theorem for functions in $\Apqo$}\label{sec3}

To prove Theorem~\ref{th:reverse} some definitions and lemmas are needed.
For each dyadic polar rectangle $Q_{j,l}$ defined in \eqref{eq:qjl}, consider the set of indexes
    $$
    U_{j,l}=\left\{(i,m): {\rm dist}(Q_{i,m},Q_{j,l})\le\frac{1}{K^{j+1}}\left(1-\frac{1}{K}\right)\right\},\quad j\in\N\cup\{0\},\quad l=0,1,\ldots,K^{j+3}-1,
    $$
and denote
    \begin{equation}\label{bigsquare}
    \widehat{Q}_{j,l}=\bigcup_{(i,m)\in U_{j,l}}Q_{i,m}.
    \end{equation}
For $f\in \H(\D)$, define $f_{j,l}=\sup_{z\in Q_{j,l}} |f(z)|$ and $\widehat{f}_{j,l}=\sup_{\z\in \widehat{Q}_{j,l}}|f(\z)|$. Then
    \begin{equation}\label{1}
    \widehat{f}_{j,l}\le\sum_{(i,m)\in U_{j,l}}f_{i,m}\lesssim\widehat{f}_{j,l},\quad j\in\N\cup\{0\},\quad l=0,1,\ldots,K^{j+3}-1,
    \end{equation}
because $\# U_{j,l}$ has a finite uniform bound independent of $j$, $l$ and $K$. For $f\in\H(\D)$, write $\lambda(f)=\{\lambda(f)_{j,l}\}$, where
    \begin{equation}\label{eq:lambda}
    \lambda(f)_{j,l}=K^{-\frac{j}{p}}\whw(r_j)^\frac{1}{q} f_{j,l}
    \end{equation}
for all $j\in\N\cup\{0\}$ and $l=0,1,\ldots,K^{j+3}-1$.

\begin{lemma}\label{th:equivnorm}
Let $0<p,q<\infty$, $\omega\in\DDD$ and $K\in\N\setminus\{1\}$ such that \eqref{K} holds.
Then $\Vert f \Vert_{\Apqo} \asymp \Vert\lambda(f)\Vert_{\ell^{p,q}}$ for all $f\in \H(\D)$.
\end{lemma}

\begin{proof}
Lemma~\ref{Lemma:replacement-Lemmas-Memoirs}(ii) implies
    \begin{equation*}
    \begin{split}
    \Vert f \Vert^q_{\Apqo}
    =\sum_{j=0}^\infty \int_{r_j}^{r_{j+1}}M_p^q(r,f)\omega(r)\,dr
    \le\sum_{j=0}^\infty  M_p^q(r_{j+1},f) \whw(r_j)
    \lesssim\sum_{j=1}^\infty  M_p^q(r_{j},f) \whw(r_j),
    \end{split}
    \end{equation*}
where
    \begin{equation*}
    \begin{split}
    M_p^p(r_{j},f)
    &= \int_0^{2\pi} \left| f(r_j e^{i\theta}) \right|^p\,d\theta
    = \sum_{l=0}^{K^{j+3}-1} \int_{\frac{2\pi l}{K^{j+3}}}^{\frac{2\pi (l+1)}{K^{j+3}}} \left| f(r_j e^{i\theta}) \right|^p\,d\theta
    \lesssim K^{-j}\sum_{l=0}^{K^{j+3}-1} f_{j,l}^p,
    \end{split}
    \end{equation*}
and hence
    \begin{equation*}
    \begin{split}
    \Vert f \Vert^q_{\Apqo}
    &\lesssim \sum_{j=1}^\infty \left( K^{-j}\sum_{l=0}^{K^{j+3}-1} f_{j,l}^p\right)^\frac{q}{p}\whw(r_j)
    =\sum_{j=1}^\infty \left( \sum_{l=0}^{K^{j+3}-1}\left(K^{-\frac{j}{p}}\whw(r_j)^\frac{1}{q} f_{j,l}\right)^p\right)^\frac{q}{p}
    \le\|\lambda(f)\|_{\ell^{p,q}}^q.
    \end{split}
    \end{equation*}

To prove the reverse inequality, choose
$z^\star_{j,l}\in\overline{Q_{j,l}}$ such that $f_{j,l}=|f(z^\star_{j,l})|$, and $n_0\in\N$ such that
    $$
    r_{j-1}\le r_j-\frac{\text{diam} Q_{j,l}}{K^{n_0}}<r_{j+1}+\frac{\text{diam} Q_{j,l}}{K^{n_0}}\le r_{j+2}
    $$
for all $j$ and $l$. Then the subharmonicity of $|f|^p$ gives
    \begin{equation*}
    \begin{split}
    \sum_{l=0}^{K^{j+3}-1}f_{j,l}^p
    &\lesssim\sum_{l=0}^{K^{j+3}-1}\frac{1}{|D(z^\star_{j,l},\frac{\text{diam} Q_{j,l}}{K^{n_0}})|}\int_{D\left(z^\star_{j,l},\frac{\text{diam} Q_{j,l}}{K^{n_0}}\right)} |f(\z)|^p\,dA(\z)\\
    &\lesssim K^{2j}\int_{A_{j-1}\cup A_{j}\cup A_{j+1}}|f(\z)|^p dA(\z)
    \lesssim K^{j} M_p^p(r_{j+2},f),\quad j\in\N\cup\{0\},
    \end{split}
    \end{equation*}
with the convenience that $A_{-1}=\emptyset$. Moreover, \eqref{K} implies $\whw(r_j)\le\frac{C}{C-1}\int_{r_j}^{r_{j+1}}\omega(r)\,dr$ for all $j\in\N\cup\{0\}$. These two estimates together with Lemma~\ref{Lemma:replacement-Lemmas-Memoirs}(ii) now yield
    \begin{equation}\label{eq:equivnorm4}
    \begin{split}
    \|\lambda(f)\|^q_{\ell^{p,q}}
    &=\sum_{j=0}^\infty \left(\sum_{l=0}^{K^{j+3}-1}\left(K^{-\frac{j}{p}}\whw(r_j)^\frac{1}{q} f_{j,l}\right)^p\right)^\frac{q}{p}
    \lesssim \sum_{j=0}^\infty\left( K^{-j} \whw(r_j)^\frac{p}{q} K^{j} M_p^p(r_{j+2},f)  \right)^\frac{q}{p}\\
    &\lesssim \sum_{j=2}^\infty \whw(r_{j}) M_p^q(r_{j},f)
    \lesssim\sum_{j=0}^\infty \int_{r_j}^{r_{j+1}}M_p^q(r,f)\omega(r)\,dr
    =\Vert f \Vert^q_{\Apqo},
    \end{split}
    \end{equation}
and therefore the assertion is proved.
\end{proof}

\begin{lemma}\label{le:a1}
Let $0<p\le \infty$, $0<q<\infty$ and $\om\in\DD$, and let $\beta=\beta(\om)>0$ be that of Lemma~\ref{Lemma:replacement-Lemmas-Memoirs}(ii). Then $A^{p,q}_\omega\subset A^1_\eta$ for all $\eta>\frac{\b}{q}+\frac{1}{p}-1$.
\end{lemma}

\begin{proof}
If $f\in\H(\D)$ and $\om\in\DD$, then the well known inequality $M_\infty(r,f)\lesssim
M_p(\frac{1+r}{2},f)(1-r)^{-1/p}$ gives
    \begin{equation*}
    \begin{split}
    \|f\|_{\Apqo}^q&\ge\int_{\frac{1+r}{2}}^1M_p^q(s,f)\omega(s)\,ds
    \ge
    M_p^q\left(\frac{1+r}{2},f\right)\widehat{\om}\left(\frac{1+r}{2}\right)
    \gtrsim M_\infty^q(r,f)\widehat{\om}(r)(1-r)^\frac{q}{p},
    \end{split}
    \end{equation*}
from which Lemma~\ref{Lemma:replacement-Lemmas-Memoirs}(ii) yields
    $$
    \|f\|_{A^1_\eta}\lesssim\|f\|_{\Apqo}\int_0^1\frac{(1-r)^{\eta-\frac1p}}{\widehat{\om}(r)^\frac1q}\,dr
    \lesssim\frac{\|f\|_{\Apqo}}{\widehat{\om}(0)^\frac1q}\int_0^1(1-r)^{\eta-\frac1p-\frac\b{q}}\,dr\asymp\|f\|_{\Apqo},
    $$
and the assertion follows.
\end{proof}

\begin{Prf}{\em{Theorem~\ref{th:reverse}}.}
The fact that the functions of the form \eqref{Reveq} with $\lambda(f)=\{\lambda(f)_{j,l}^k\}\in\ell^{p,q}$ belong to $A^{p,q}_\omega$ and the inequality
    $
   \left\|\{\lambda(f)_{j,l}^k\} \right\|_{\ell^{p,q}}\lesssim\Vert f\Vert_{A^{p,q}_\omega}
    $
follow from Theorem~\ref{th:atomicdecomx}.

Let us now prove that each function in $\Apqo$ is of the form \eqref{Reveq}, where $\lambda(f)=\{\lambda(f)_{j,l}^k\}\in\ell^{p,q}$, and
    $
   \Vert f\Vert_{A^{p,q}_\omega}\lesssim\left\|\{\lambda(f)_{j,l}^k\} \right\|_{\ell^{p,q}}.
    $
To do this we use ideas from \cite[(1.5) Theorem]{RicciTaibleson1983}. Let $\eta=\eta(p,q,\om)>1+\frac{1}{p}+\frac{\b+\gamma}{q}$, where $\beta=\beta(\om)>0$ and $\gamma=\gamma(\om)>0$ are those of Lemma~\ref{Lemma:replacement-Lemmas-Memoirs}(ii)(iii). Then $\eta$ satisfies the condition \eqref{eq:M} assumed on $M$ in Theorem~\ref{th:atomicdecomx}, and $A^{p,q}_\omega\subset A^1_\eta$ by Lemma~\ref{le:a1}. Therefore, in particular,
    \begin{equation*}
    P_\eta(f)(z)=(\eta+1)\int_{\D}\frac{f(\z)}{(1-\overline{\z}z)^{2+\eta}}(1-|\z|^2)^{\eta}dA(\z)=f(z),\quad z\in\D,\quad f\in A^{p,q}_\omega.
    \end{equation*}
Consider the operator $S_\eta$ defined by
    \begin{equation*}
    \begin{split}
    S_\eta(f)(z)
    &=(\eta+1)\sum_{j,l,k}f(\z_{j,l}^k)\frac{(1-|\z_{j,l}^k|^2)^{\eta}}{\left(1-\overline{\z_{j,l}^k}z\right)^{\eta+2}}\left|Q_{j,l}^k \right|\\
    &=(\eta+1)\sum_{j,l,k}f(\z_{j,l}^k)(1-|\z_{j,l}^k|^2)^{\frac{1}{p}}\whw(r_j)^{\frac{1}{q}}
    \frac{(1-|\z_{j,l}^k|^2)^{\eta-\frac{1}{p}}\whw(r_j)^{-\frac{1}{q}}}{\left(1-\overline{\z_{j,l}^k}z\right)^{\eta+2}}\left|Q_{j,l}^k \right|\\
    &=(\eta+1) \sum_{j,l,k}a_{j,l,k}(f)
    \frac{(1-|\z_{j,l}^k|^2)^{\eta-\frac{1}{p}}\whw(r_j)^{-\frac{1}{q}}}{\left(1-\overline{\z_{j,l}^k}z\right)^{\eta+2}}\left|Q_{j,l}^k \right|,\quad z\in\D,
    \end{split}
    \end{equation*}
where
    \begin{equation}\label{a}
    a(f)_{j,l,k}= f(\z_{j,l}^k)(1-|\z_{j,l}^k|^2)^{\frac{1}{p}}\whw(r_j)^{\frac{1}{q}},\quad j\in\N\cup\{0\},\quad l=0,\ldots,K^j-1,\quad k=1,\ldots,M^2.
    \end{equation}
The estimate $M_\infty(r,f)\lesssim\|f\|_{A^{p,q}_\om}\widehat{\om}(r)^{-\frac1q}(1-r)^{-\frac1p}$, obtained in the proof of Lemma~\ref{le:a1}, ensures that $S_\eta(f)$ is well defined for each $f\in A^{p,q}_\om$. Write $a(f)=\{a(f)_{j,l,k}\}$ and observe that
\begin{equation}\label{eq:iter}
 \Vert a(f)\Vert_{\ell^{p,q}}\lesssim\|\lambda(f)\|_{\ell^{p,q}}\asymp\Vert f \Vert_{\Apqo}
\end{equation}
by Lemma~\ref{th:equivnorm}. It is shown next that for $M$ large enough, $S_\eta$ satisfies
    \begin{equation}\label{enum13}
    \Vert f-S_\eta(f) \Vert_{\Apqo}\le\frac{1}{2}\Vert f \Vert_{\Apqo}.
    \end{equation}
To see this, note first that
    \begin{equation}\label{P-S}
    \begin{split}
    f(z)-S_\eta(f)(z)
    &=P_\eta(f)(z)-S_\eta(f)(z) \\
    &=(\eta+1)\Bigg(\int_{\D}\frac{f(\z)}{(1-\overline{\z}z)^{2+\eta}}(1-|\z|^2)^{\eta}dA(\z)\\
    &\quad-\sum_{j,l,k}f(\z_{j,l}^k)\frac{\left(1-|\z_{j,l}^k|^2\right)^\eta}{\left(1-\overline{\z_{j,l}^k}z\right)^{\eta+2}}\left|Q_{j,l}^k \right|\Bigg)\\
    &=(\eta+1)\sum_{j,l,k}\int_{Q_{j,l}^k }\left(H_z(\z)-H_z(\z_{j,l}^k)\right)dA(\z),\quad z\in\D,
    \end{split}
    \end{equation}
where $H_z(\z)=f(\z)\frac{(1-|\z|^2)^\eta}{(1-\overline{\z}z)^{\eta+2}}$ for all $z,\z\in\D$. It is clear that $H_z$ satisfies
    \begin{equation}\label{estimate}
    \begin{split}
    \left|H_z(\z)-H_z(\z_{j,l}^k) \right|
    &\le\text{diam } Q_{j,l}^k \sup_{w \in Q_{j,l}^k} \left|\nabla H_z(w) \right|,\quad \z\in Q_{j,l}^k,
    \end{split}
    \end{equation}
and also
    \begin{equation}\label{deriv1}
    \begin{split}
    \frac{\partial}{\partial \z} H_z(\z)=\left( f'(\z)(1-|\z|^2)-f(\z)\eta\overline{\z}\right)\frac{(1-|\z|^2)^{\eta-1}}{(1-\overline{\z}z)^{\eta+2}},\quad \z\in\D,
    \end{split}
    \end{equation}
and
    \begin{equation}\label{deriv2}
    \begin{split}
    \frac{\partial}{\partial \overline{\z}}H_z(\z)
    = f(\z)\frac{(1-|\z|^2)^{\eta-1}}{(1-\overline{\z}z)^{\eta+2}}\left( (\eta+2)z\frac{1-|\z|^2}{1-\overline{\z}z}-\eta\z\right),\quad \z\in\D.
    \end{split}
    \end{equation}
Moreover, the Cauchy integral formula implies
    \begin{equation}\label{Cauchy-integral}
    \begin{split}
    |f^{(n)}(\z)|
    \lesssim \int_{|\xi-\z|=\frac{1}{K^{j+1}\left(1-\frac1K\right)}}\frac{|f(\xi)|}{|\xi-\z|^{n+1}}\,|d\xi|
    \lesssim K^{jn}\widehat{f}_{j,l},\quad \z \in Q_{j,l},\quad n\in \N\cup \{0\}.
    \end{split}
    \end{equation}
The identities \eqref{deriv1} and  \eqref{deriv2} together with the estimate \eqref{Cauchy-integral} now give
    \begin{equation}\label{gradient}
    \begin{split}
    \sup_{w \in Q_{j,l}^k}\left|\nabla H_z(w) \right|
    &\lesssim  \frac{(1-|\z_{j,l}^k|^2)^{\eta-1}}{\left|1-z\overline{\z_{j,l}^k}\right|^{\eta+2}}  \sup_{w \in \widehat{Q}_{j,l}} |f(w)|
    =\frac{(1-|\z_{j,l}^k|^2)^{\eta-1}}{\left|1-z\overline{\z_{j,l}^k}\right|^{\eta+2}}\widehat{f}_{j,l}.
    \end{split}
    \end{equation}
The identity \eqref{P-S} together with the estimates \eqref{estimate}, \eqref{gradient} and \eqref{1} give
    \begin{equation*}
    \begin{split}
    |f(z)-S_\eta(f)(z)|
    &\lesssim\sum_{j,l,k}\left(\int_{Q_{j,l}^k}\left(\text{diam } Q_{j,l}^k \frac{(1-|\z^k_{j,l}|^2)^{\eta-1}}{\left|1-z\overline{\z^k_{j,l}}\right|^{\eta+2}}\sup_{w\in\widehat{Q}_{j,l}} |f(w)|\right)dA(\z)\right)\\
    &\lesssim\sum_{j,l}\sum_{k=1}^{M^2}\left(\left|Q_{j,l}^k\right|  \text{diam } Q_{j,l}^k \frac{(1-|\z^k_{j,l}|^2)^{\eta-1}}{\left|1-z\overline{\z^k_{j,l}}\right|^{\eta+2}} \widehat{f}_{j,l} \right)\\
    &\lesssim\sum_{j,l} \frac{1}{M^3} \frac{(1-|\z_{j,l}|^2)^{\eta+2}}{\left|1-z\overline{\z_{j,l}}\right|^{\eta+2}} \widehat{f}_{j,l}\sum_{k=1}^{M^2} 1  \\
    & \le  \frac{1}{M}\sum_{j,l}  \frac{(1-|\z_{j,l}|^2)^{\eta+2}}{\left|1-z\overline{\z_{j,l}}\right|^{\eta+2}} 
    \sum_{(i,m)\in U_{k,j}}f_{i,m}\\
    &\lesssim \frac{1}{M}\sum_{j,l}\lambda(f)_{j,l}\frac{(1-|\z_{j,l}|^2)^{\eta+2-\frac{1}{p}}\whw(r_j)^{-\frac{1}{q}}}{\left|1-z\overline{\z_{j,l}}\right|^{\eta+2}},
    \end{split}
    \end{equation*}
where $\{\lambda_{j,l}\}$ is that of \eqref{eq:lambda}. Then, by combining the above estimate with Theorem~\ref{th:atomicdecomx}, with $M=\eta+2$, and Lemma~\ref{th:equivnorm} it follows that
    \begin{equation*}
    \begin{split}
    \Vert f -S_\eta(f)\Vert_{\Apqo} \lesssim  \frac{1}{M} \Vert \lambda(f) \Vert_{\ell^{p,q}} \asymp  \frac{1}{M} \Vert f\Vert_{\Apqo}.
    \end{split}
    \end{equation*}
The inequality \eqref{enum13} follows by choosing $M$ large enough.

Let $\{f_n\}_{n=1}^\infty$ be defined by $f_1=S_\eta(f)$ and $f_n=S_\eta\left(f-\sum_{m=1}^{n-1}f_m \right)$ for $n\in\N\setminus\{1\}$. Further, let $a(f)^{(1)}_{j,l,k}=a(f)_{j,l,k}$, where $\{a(f)_{j,l,k}\}$ are those defined in \eqref{a}, and
    $$
    a(f)^{(n)}_{j,l,k}= \left(f-\sum_{m=1}^{n-1} f_{m}\right)(\z_{j,l}^k)(1-|\z_{j,l}^k|^2)^{\frac{1}{p}}\whw(r_j)^{\frac{1}{q}},\quad n\in\N\setminus\{1\}.
    $$
With this notation
    \begin{equation}\label{enum21}
    f_n(z)=(\eta+1) \sum_{j,l,k} a(f)^{(n)}_{j,l,k}\frac{(1-|\z_{j,l}^k|^2)^{\eta-\frac{1}{p}}\whw(r_j)^{-\frac{1}{q}}}{\left(1-\overline{\z_{j,l}^k}z\right)^{\eta+2}}\left|Q_{j,l}^k \right|,\quad n\in\N,
    \end{equation}
by the definition of $S_\eta$. Moreover, $n$ applications of \eqref{enum13} give
    \begin{equation}\label{enum23}
    \begin{split}
    \left\| f-\sum_{m=1}^n f_m \right\|_{\Apqo}
    &=\left\|  f-\sum_{m=1}^{n-1} f_m -f_n\right\|_{\Apqo}
    =\left\|  \left(Id-S_\eta\right) \left( f-\sum_{m=1}^{n-1} f_m\right)\right\|_{\Apqo}\\
    &\leq \frac{1}{2}\left\| \left( f-\sum_{m=1}^{n-1} f_m\right)\right\|_{\Apqo}
    \leq \cdots
    \leq \frac{1}{2^n}\left\| f \right\|_{\Apqo}.
    \end{split}
    \end{equation}
Therefore, by denoting $a(f)^{(n)}=\{a(f)^{(n)}_{j,l,k}\}$, and applying
\eqref{eq:iter} to $f-\sum_{m=1}^{n-1}f_m$
 yields
    \begin{equation}\label{enum22}
    \begin{split}
    \Vert a(f)^{(n)} \Vert_{\ell^{p,q}} \lesssim \left\| f-\sum_{m=1}^{n-1}f_m  \right\|_{\Apqo} \le 2^{-n+1}\Vert f \Vert_{\Apqo},\quad n\in\N.
    \end{split}
    \end{equation}

Finally, set $b(f)_{j,l,k}=\sum_{n=1}^\infty a(f)^{(n)}_{j,l,k}$ and
    $$
    g(z)=(\eta+1) \sum_{j,l,k} b(f)_{j,l,k}\frac{(1-|\z_{j,l}^k|^2)^{\eta-\frac{1}{p}}\whw(r_j)^{-\frac{1}{q}}}{(1-\overline{\z_{j,l}^k}z)^{\eta+2}}\left|Q_{j,l}^k \right|,\quad z\in\D.
    $$
Then \eqref{enum21} yields
    \begin{equation*}
    \begin{split}
    g(z)-\sum_{m=1}^n f_m
    &=(\eta+1) \sum_{j,l,k} \left(b(f)_{j,l,k}- \sum_{m=1}^{n}a(f)^{(m)}_{j,l,k}\right) \frac{(1-|\z_{j,l}^k|^2)^{\eta-\frac{1}{p}}\whw(r_j)^{-\frac{1}{q}}}{(1-\overline{\z_{j,l}^k}z)^{\eta+2}}\left|Q_{j,l}^k \right|\\
    &=\sum_{j,l,k}\left(\sum_{m=n+1}^{\infty}a(f)^{(m)}_{j,l,k}\right) \frac{(1-|\z_{j,l}^k|^2)^{\eta-\frac{1}{p}}\whw(r_j)^{-\frac{1}{q}}}{(1-\overline{\z_{j,l}^k}z)^{\eta+2}}\left|Q_{j,l}^k\right|,
    \end{split}
    \end{equation*}
from which Theorem~\ref{th:atomicdecomx} and \eqref{enum22} give
    \begin{equation*}
    \begin{split}
    \left\|g-\sum_{m=1}^n f_m\right\|_{\Apqo}
    &\lesssim\left\| \left\{\sum_{m=n+1}^{\infty}a^{(m)}_{j,l,k}\right\}_{j,l,k}\right\|_{\ell^{p,q}}
    \lesssim\left\| f \right\|_{\Apqo}\left(\sum_{m=n+1}^\infty 2^{-m\min\{1,p,q\}}\right)^\frac1{\min\{1,p,q\}}\\
    &\asymp2^{-n}\left\| f \right\|_{\Apqo},\quad n\in\N.
    \end{split}
    \end{equation*}
By combining this with \eqref{enum23} we deduce
    \begin{equation}\label{P-S-6}
    \begin{split}
    \left\| f-g \right\|_{\Apqo}
    &\leq \left\| f-\sum_{m=1}^n f_m+\sum_{m=1}^n f_m-g \right\|_{\Apqo}
    \\ & \lesssim \left\| f-\sum_{m=1}^n f_m \right\|_{\Apqo} + \left\| g- \sum_{m=1}^n f_m\right\|_{\Apqo} \\
    &\lesssim  2^{-n}\left\| f \right\|_{\Apqo},\quad n\in\N\setminus\{1\},
    \end{split}
    \end{equation}
and it follows that $f=g$. The assertion of the theorem follows for $M=\eta+2$ and
    $$
    \lambda(f)_{j,l}^k=(\eta+1)b(f)_{j,l,k}\frac{|Q_{j,l}^k|}{(1-|\z_{j,l}^k|^2)^2},
    $$
because $\Vert \{\lambda(f)_{j,l}^k\} \Vert_{\ell^{p,q}}\lesssim \left\| f \right\|_{\Apqo}$ by \eqref{enum22}. This finishes the proof.
 \end{Prf}

\section{Differentiation operators from $A^{p,q}_\om$ to $L^s_\mu$}\label{sec5}

We recall that the spaces $\ell^{p,q}$ obey the basic inclusion relations $\ell^{p,q} \subset \ell^{r,q}$ for $p\le r$, and $\ell^{p,q} \subset \ell^{p,s}$ if $q\le s$. Moreover, it is known that by denoting
    \begin{equation*}
    p'=\left\{
        \begin{array}{cl}
        \infty,\, & 0<p\le 1,\\
        \frac{p}{p-1},\, & 1<p<\infty,\\
        1,\,&p=\infty,
        \end{array}\right.
    \end{equation*}
we have
    $$
    \left\| b\right\|_{\ell^{p',q'}}=\sup \left\{ \left|\sum_{j,l} c_{j,l}b_{j,l} \right|: \left\|c \right\|_{\ell^{p,q}}=1\right\}
    $$
by \cite[Theorem~1]{Nakamura}. The following proof uses ideas form the proof of \cite[Theorem~2]{Luecking}.

\medskip

\begin{Prf}{\em{Theorem~\ref{Theorem:Carleson}.}}
Assume first that $D^{(n)}:\Apqo \to L^s_\mu$ is bounded. Let
    $$
    F_t(z)=\sum_{j,l}a_{j,l}(t)\lambda_{j,l}\frac{(1-|z_{j,l}|)^{M-\frac{1}{p}}\omg(z_{j,l})^{-\frac{1}{q}}}{\left(1-\overline{z_{j,l}}z\right)^M},\quad z\in\D,
    $$
where $\{z_k\}$ is a separated sequence and $a_{j,l}$ are the Rademacher functions \cite[Appendix~A]{Durenhp} and $M$ satisfies the hypothesis \eqref{eq:M} of Theorem~\ref{th:atomicdecomx}. Then
    \begin{equation}\label{eq:CMt}
    \Vert F_t^{(n)}\Vert_{L^s_\mu}
    \le\Vert D^{(n)} \Vert_{\Apqo\to L^s_\mu}\Vert F_t \Vert_{\Apqo}
    \lesssim\Vert D^{(n)} \Vert_{\Apqo\to L^s_\mu}\Vert \lambda \Vert_{\ell^{p,q}}.
\end{equation}
Moreover, Khinchine's inequality~\cite[Appendix~A]{Durenhp} yields
    \begin{equation}\label{eq:CMk}
    \begin{split}
   &\int_{0}^{1} \Vert F_t^{(n)}\Vert^s_{L^s_\mu} \,dt
    \\ &=\int_{\D} \int_{0}^{1}\left|M(M+1)\cdots(M+n-1)\sum_{j,l} a_{j,l}(t)\lambda_{j,l}\overline{z_{j,l}}^n\frac{(1-|z_{j,l}|)^{M-\frac{1}{p}}\omg(z_{j,l})^{-\frac{1}{q}}}{\left(1-\overline{z_{j,l}}z\right)^{M+n}}\right|^s \,dt\, d\mu(z) \\
    &\gtrsim \int_{\D} \left(\sum_{j,l} \left|\lambda_{j,l}\frac{(1-|z_{j,l}|)^{M-\frac{1}{p}}\omg(z_{j,l})^{-\frac{1}{q}}}{\left(1-\overline{z_{j,l}}z\right)^{M+n}}\right|^2\right)^\frac{s}{2} d\mu(z)\\
    &\gtrsim \int_{\D} \left(\sum_{j,l} \left|\chi_{Q_{j,l}}(z)\lambda_{j,l}(1-|z_{j,l}|)^{-n-\frac{1}{p}}\omg(z_{j,l})^{-\frac{1}{q}}\right|^2\right)^\frac{s}{2} d\mu(z)\\
    &=\sum_{j,l}|\lambda_{j,l}|^s\mu(Q_{j,l})(1-|z_{j,l}|)^{-s\left(n+\frac{1}{p}\right)}\omg(z_{j,l})^{-\frac{s}{q}}\\
    &\asymp \sum_{j,l}|\lambda_{j,l}|^s\mu(Q_{j,l}) K^{js\left(n+\frac{1}{p}\right)}\omg(r_j)^{-\frac{s}{q}}.
    \end{split}
    \end{equation}
By integrating \eqref{eq:CMt} with respect to $t$, using \eqref{eq:CMk} and writing $b=\{b_{j,l}\}=\{|\lambda_{j,l}|^s\}$ we obtain
    \begin{equation*}
    \sum_{j,l} b_{j,l} \mu(Q_{j,l}) K^{js\left(n+\frac{1}{p}\right)}\omg(r_j)^{-\frac{s}{q}}
    \lesssim\Vert D^{(n)} \Vert^s_{\Apqo\to L^s_\mu}\Vert b \Vert_{\ell^{\left(\frac{p}{s}\right),\left(\frac{q}{s}\right)}}
    \end{equation*}
for all $b\in \ell^{\left(\frac{p}{s}\right),\left(\frac{q}{s}\right)}$ with $b_{j,l}\ge0$. It follows that
    $$
    \left\{\mu(Q_{j,l})K^{sj(n+\frac{1}{p})}\whw(r_j)^{\frac{-s}{q}}\right\}_{j,l}\in \ell^{\left(\frac{p}{s}\right)',\left(\frac{q}{s}\right)'}
    $$
with norm bounded by a constant times $\Vert D^{(n)} \Vert^s_{\Apqo\to L^s_\mu}$ by \cite[Theorem~1]{Nakamura}. Thus (ii) is satisfied.

To see the converse implication, note first that the estimate \eqref{Cauchy-integral} implies $f^{(n)}_{j,l}\lesssim K^{jn}\widehat{f}_{j,l}$. This together with the fact that $\# U_{j,l}$ has a finite uniform bound independent of $j$, $l$ and $K$, and Lemma~\ref{th:equivnorm} give
    \begin{equation*}
    \begin{split}
    \left\|\left\{f^{(n)}_{j,l}K^{-j\left(n+\frac1p\right)}\whw(r_j)^\frac{1}{q} \right\}_{j,l}\right\|_{\ell^{p,q}}
    &\lesssim\left\|\left\{\widehat{f}_{j,l}K^{-\frac{j}{p}}\whw(r_j)^\frac{1}{q}\right\}_{j,l}\right\|_{\ell^{p,q}}
    \lesssim\left\|\left\{f_{j,l}K^{-\frac{j}{p}}\whw(r_j)^\frac{1}{q}\right\}_{j,l}\right\|_{\ell^{p,q}}\\
    &=\left\|\lambda(f)\right\|_{\ell^{p,q}}\asymp\|f\|_{A^{p,q}_\om}.
    \end{split}
    \end{equation*}
By applying \cite[Theorem~1]{Nakamura} and the estimate just established, we deduce
    \begin{equation*}
    \begin{split}
    \int_\D|f^{(n)}(z)|^s\,d\mu(z)
    &=\sum_{j,l}\int_{Q_{j,l}}|f^{(n)}(z)|^s\,d\mu(z)
    \le\sum_{j,l}\left(f_{j,l}^{(n)}\right)^s\mu(Q_{j,l})\\
    &=\sum_{j,l}\left(f_{j,l}^{(n)}\right)^sK^{-js\left(n+\frac{1}{p}\right)}\whw(r_j)^\frac{s}{q}\mu(Q_{j,l})K^{js\left(n+\frac{1}{p}\right)}\whw(r_j)^{-\frac{s}{q}}\\
    &\le\left\|\left\{\left(f_{j,l}^{(n)}\right)^sK^{-js\left(n+\frac{1}{p}\right)}\whw(r_j)^\frac{s}{q}\right\}_{j,l}\right\|_{\ell^{\frac{p}{s},\frac{q}{s}}}\\
    &\quad\cdot\left\|\left\{\mu(Q_{j,l})K^{js\left(n+\frac{1}{p}\right)}\whw(r_j)^{-\frac{s}{q}}\right\}_{j,l}\right\|_{\ell^{\left(\frac{p}{s}\right)',\left(\frac{q}{s}\right)'}}\\
    &=\left\|\left\{f_{j,l}^{(n)}K^{-j\left(n+\frac{1}{p}\right)}\whw(r_j)^\frac{1}{q}\right\}_{j,l}\right\|_{\ell^{p,q}}^s\\
    &\quad\cdot\left\|\left\{\mu(Q_{j,l})K^{js\left(n+\frac{1}{p}\right)}\whw(r_j)^{-\frac{s}{q}}\right\}_{j,l}\right\|_{\ell^{\left(\frac{p}{s}\right)',\left(\frac{q}{s}\right)'}}\\
    &\lesssim \Vert f\Vert^s_{\Apqo}
    \left\|\left\{\mu(Q_{j,l})K^{js\left(n+\frac{1}{p}\right)}\whw(r_j)^{-\frac{s}{q}}\right\}_{j,l}\right\|_{\ell^{\left(\frac{p}{s}\right)',\left(\frac{q}{s}\right)'}},
    \end{split}
    \end{equation*}
and hence
    $$
    \|D^{(n)}\|^s_{\Apqo\to L^s_\mu}
    \lesssim \left\|\left\{\mu(Q_{j,l})2^{sj\left(n+\frac{1}{p}\right)}\whw(r_j)^{-\frac{s}{q}}\right\}_{j,l} \right\|_{\ell^{\left(\frac{p}{s}\right)',\left(\frac{q}{s}\right)'}}.
    $$

It remains to show that (ii) is equivalent to its continuous counterpart (iii). To do this, first define $U_{j,l}^r=\left\{(i,m):\varrho(Q_{j,l},Q_{i,m})<r\right\}$, and note that $\sup_{j,l}\#U_{j,l}^r\le C(r)<\infty$. Further, set $\widehat{Q}^r_{j,l}=\cup_{(i,m)\in U^r_{j,l}}Q_{i,m}$ and write $T_r=T_{r,u,v}$ for short. Assume first $s<\min\{p,q\}$. Then, by choosing $r=r(K)>0$ sufficiently large we have
    \begin{equation*}
    S_r(z)=\sum_j \sum_l \frac{\mu(Q_{j,l})}{K^{-j(sn+1)}\whw(r_j)}\chi_{Q_{j,l}}(z) \lesssim T_r(z) \lesssim \sum_j \sum_l \frac{\mu(\widehat{Q}^r_{j,l})}{K^{-j(sn+1)}\whw(r_j)}\chi_{Q_{j,l}}(z)=B_r(z)
    \end{equation*}
for all $z\in\D$. By using $\sup_{j,l}\#U_{j,l}^r<\infty$ we deduce
    \begin{equation*}
    \begin{split}
    \left\|B_r \right\|^{\frac{q}{q-s}}_{L^{\frac{p}{p-s},\frac{q}{q-s}}_\omega}
    &=\int_0^1 \left( \int_0^{2\pi} B_r(te^{i\theta})^{\frac{p}{p-s}}d\theta \right)^{\frac{q(p-s)}{p(q-s)}}\omega(t)\,dt\\
    &\asymp\sum_{k=0}^\infty \int_{r_k}^{r_{k+1}} \left( \sum_{i=0}^{K^{k+3}-1}  \mu(\widehat{Q}^r_{k,i})^{\frac{p}{p-s}}K^{k\frac{p}{p-s}(1+sn)}K^{-k}\widehat{\omega}(r_k)^{-\frac{p}{p-s}} \right)^{\frac{q(p-s)}{p(q-s)}}\omega(t)\,dt\\
    &\asymp\sum_{k=0}^\infty\left(\sum_{i=0}^{K^{k+3}-1}\left(\mu(\widehat{Q}^r_{k,i}) K^{ks(n+\frac{1}{p})}\widehat{\omega}(r_k)^{-\frac{s}{q}}\right)^\frac{p}{p-s} \right)^{\frac{q(p-s)}{p(q-s)}}\\
    &\lesssim\sum_{k=0}^\infty\left(\sum_{i=0}^{K^{k+3}-1}\left(\mu(Q_{k,i}) K^{ks(n+\frac{1}{p})}\widehat{\omega}(r_k)^{-\frac{s}{q}}\right)^\frac{p}{p-s} \right)^{\frac{q(p-s)}{p(q-s)}},
    \end{split}
    \end{equation*}
and an essentially identical reasoning gives
    \begin{equation*}
    \begin{split}
    \left\|S_r \right\|^{\frac{q}{q-s}}_{L^{\frac{p}{p-s},\frac{q}{q-s}}_\omega}
    &\asymp\sum_{k=0}^\infty\left(\sum_{i=0}^{K^{k+3}-1}\left(\mu(Q_{k,i}) K^{ks(n+\frac{1}{p})}\omega(r_k)^{-\frac{s}{q}}\right)^\frac{p}{p-s} \right)^{\frac{q(p-s)}{p(q-s)}}.
    \end{split}
    \end{equation*}
This finishes the proof of the case $s<\min\{p,q\}$ for $r>0$ large enough.

In the case $p\leq s<q$, we have
    \begin{equation*}
    \begin{split}
    S_r(z)
    =\sum_j \sum_l \frac{\mu(Q_{j,l})}{K^{-js(n+\frac{1}{p})}\whw(r_j)}\chi_{Q_{j,l}}(z)
    \lesssim T_r(z)
    \lesssim \sum_j \sum_l \frac{\mu(\widehat{Q}^r_{j,l})}{K^{-js(n+\frac{1}{p})}\whw(r_j)}\chi_{Q_{j,l}}(z)=B_r(z)
    \end{split}
    \end{equation*}
for all $z\in\D$. By using again $\sup_{j,l}\#U_{j,l}^r<\infty$ we deduce
    \begin{equation*}
    \begin{split}
    \left\|B_r \right\|_{L^{\infty,\frac{q}{q-s}}_\omega}
    &= \left(\int_0^1 \left( M_\infty(B_r,t)  \right)^{\frac{q}{q-s}}\omega(t)\,dt \right)^{\frac{q-s}{q}}\\
    &\asymp \left(\sum_{k=0}^\infty \int_{r_k}^{r_{k+1}} \left( K^{ks(n+\frac{1}{p})}\whw(r_k)^{-1}\sup_i  \mu(\widehat{Q}^r_{k,i}) \right)^{\frac{q}{q-s}}\omega(t)\,dt\right)^{\frac{q-s}{q}}\\
    &\asymp \left(\sum_{k=0}^\infty\left( \sup_i  \mu(\widehat{Q}^r_{j,i}) K^{ks(n+\frac{1}{p})}\widehat{\omega}(r_k)^{-\frac{s}{q}}\right)^{\frac{q}{q-s}}\right)^{\frac{q-s}{q}}\\
    &\lesssim \left(\sum_{k=0}^\infty\left( \sup_i  \mu(Q_{j,i}) K^{ks(n+\frac{1}{p})}\omega(r_k)^{-\frac{s}{q}}\right)^{\frac{q}{q-s}}\right)^{\frac{q-s}{q}},
    \end{split}
    \end{equation*}
and similarly
    $$
    \left\|S_r \right\|_{L^{\infty,\frac{q}{q-s}}_\omega}
    \asymp \left(\sum_{k=0}^\infty\left( \sup_i  \mu(Q_{j,i}) K^{ks(n+\frac{1}{p})}\omega(r_k)^{-\frac{s}{q}}\right)^{\frac{q}{q-s}}\right)^{\frac{q-s}{q}}.
    $$
Hence the case $p\leq s<q$ is proved for $r>0$ large enough.

In the case $q\leq s<p$ we have
    \begin{equation*}
    \begin{split}
    S_r(z)
    =\sum_j \sum_l \frac{\mu(Q_{j,l})}{K^{-j(sn+1)}\whw(r_j)^{\frac{s}{q}}}\chi_{Q_{j,l}}(z)
    \lesssim T_r(z)\lesssim \sum_j \sum_l \frac{\mu(\widehat{Q}^r_{j,l})}{K^{-j(sn+1)}\whw(r_j)^{\frac{s}{q}}}\chi_{Q_{j,l}}(z)=B_r(z)
    \end{split}
    \end{equation*}
for all $z\in\D$. By using again $\sup_{j,l}\#U_{j,l}^r<\infty$ we deduce
    \begin{equation*}
    \begin{split}
    \left\|B_r \right\|_{L^{\frac{p}{p-s},\infty}_\omega}
    &=\sup_{0<t<1} \left( \int_0^{2\pi} B_r(te^{i\theta})^{\frac{p}{p-s}}d\theta \right)^{\frac{p-s}{p}}\\
    &\asymp \sup_j  \left(\sum_{i=0}^{K^{k+3}-1} \mu(\widehat{Q}^r_{j,i})^{\frac{p}{p-s}}K^{j\frac{p}{p-s}(sn+1)}K^{-j}\widehat{\omega}(r_j)^{-\frac{sp}{q(p-s)}} \right)^{\frac{p-s}{p}}\\
    &\asymp  \sup_j  \left(\sum_{i=0}^{K^{k+3}-1}\left(\mu(\widehat{Q}^r_{j,i}) K^{js(n+\frac{1}{p})}\widehat{\omega}(r_j)^{-\frac{s}{q}}\right)^\frac{p}{p-s} \right)^{\frac{p-s}{p}}\\
    &\lesssim  \sup_j  \left(\sum_{i=0}^{K^{k+3}-1}\left(\mu(Q_{j,i}) K^{js(n+\frac{1}{p})}\omega(r_j)^{-\frac{s}{q}}\right)^\frac{p}{p-s} \right)^{\frac{p-s}{p}}
    \asymp\left\|S_r \right\|_{L^{\frac{p}{p-s},\infty}_\omega},
    \end{split}
    \end{equation*}
completing the proof of the case $q\leq s<p$ for $r>0$ large enough.

The remaining case $s\ge\max\{p,q\}$ is the simplest one of all because now
    \begin{equation*}
    \begin{split}
    S_r(z)
    =\sum_j \sum_l \frac{\mu(Q_{j,l})}{K^{-js(n+\frac{1}{p})}\whw(r_j)^{\frac{s}{q}}}\chi_{Q_{j,l}}(z)
    \lesssim T_r(z)\lesssim \sum_j \sum_l \frac{\mu(\widehat{Q}^r_{j,l})}{K^{-js(n+\frac{1}{p})}\whw(r_j)^{\frac{s}{q}}}\chi_{Q_{j,l}}(z)=B_r(z),
    \end{split}
    \end{equation*}
and hence
    \begin{equation*}
    \begin{split}
    \left\|\left\{\frac{\mu(Q_{j,l})}{K^{-js(n+\frac{1}{p})}\whw(r_j)^{\frac{s}{q}}}\right\}_{j,l}\right\|_{\ell^\infty}
    &\lesssim \left\| T_r \right\|_{L^\infty}
    \lesssim \left\|\left\{\frac{\mu(\widehat{Q}^r_{j,l})}{K^{-js(n+\frac{1}{p})}\whw(r_j)^{\frac{s}{q}}}\right\}_{j,l}\right\|_{\ell^\infty}\\
    &\lesssim \left\|\left\{\frac{\mu(Q_{j,l})}{K^{-js(n+\frac{1}{p})}\whw(r_j)^{\frac{s}{q}}}\right\}_{j,l}\right\|_{\ell^\infty}.
    \end{split}
    \end{equation*}
This completes the proof of the theorem in the case in which $r>0$ is sufficiently large, say $r\ge r_0=r_0(K)$. If $r\in(0,r_0)$, then dividing $Q_{j,l}$ into $M^2$ rectangles of equal area as in the proof of Theorem~\ref{th:reverse}, and then slightly modifying the proof just presented, the assertion easily follows. The details of this deduction do not offer us anything new and are therefore omitted.
\end{Prf}

\end{document}